\documentclass[12pt]{article}
\usepackage{graphicx}
\usepackage{amscd, amssymb, amsmath}
\usepackage{enumerate}
\usepackage{hyperref}
\usepackage{lscape}
\newenvironment{eq}{\begin{equation}}{\end{equation}}
\newenvironment{proof}{{\bf Proof}:}{\vskip 5mm }

\newtheorem{proposition}{Proposition}[subsection]
\newtheorem{lemma}[proposition]{Lemma}
\newtheorem{definition}[proposition]{Definition}

\newtheorem{remark}[proposition]{Remark}

\newtheorem{problem}[proposition]{Problem}
\newtheorem{construction}[proposition]{Construction}
\newcommand{\llabel}[1]{\label{#1}}
\newcommand{\comment}[1]{}
\newcommand{\sr}{\rightarrow}

\newcommand{\nn}{{\bf N\rm}}

\newcommand{\wt}{\widetilde}

\newcommand{\wh}{\widehat}

{\vskip
3mm}

\newcommand{\spc}{{\,\,\,\,\,\,\,}}

\newcommand{\wtOb}{{\wt{\mathcal Ob}}}
\newcommand{\Ob}{{\mathcal Ob}}
\newcommand{\wOb}{\wt{\Ob}}

\newcommand\PPi{{\bf \Pi}}

\begin{document}
\parskip = 2mm
\begin{center}
{\bf\Large Products of families of types and $(\Pi,\lambda)$-structures on C-systems\footnote{\em 2000 Mathematical Subject Classification: 
03F50, 
18C50  
03B15, 
18D15, 
}}


\vspace{3mm}

{\large\bf Vladimir Voevodsky}\footnote{School of Mathematics, Institute for Advanced Study,
Princeton NJ, USA. e-mail: vladimir@ias.edu}
\vspace {3mm}

\end{center}
\begin{abstract}
In this paper we continue, following the pioneering works by J. Cartmell and T.\,Streicher, the study of the most important structures on C-systems, the structures that correspond, in the case of the syntactic C-systems, to the $(\Pi,\lambda,app,\beta,\eta)$-system of inference rules. 

One such structure was introduced by J. Cartmell and later studied by T. Streicher under the name of the products of families of types. 

We introduce the notion of a $(\Pi,\lambda)$-structure and construct a bijection, for a given C-system, between the set of $(\Pi,\lambda)$-structures and the set of Cartmell-Streicher structures. In the following paper we will show how to construct, and in some cases fully classify, the $(\Pi,\lambda)$-structures on the C-systems that correspond to universe categories.

The first section of the paper provides careful proofs of many of the properties of general C-systems.

Methods of the paper are fully constructive, that is, neither the axiom of excluded middle nor the axiom of choice are used.
\end{abstract}

\tableofcontents

\numberwithin{equation}{subsection}

\subsection{Introduction}
\label{Sec.1}

The concept of a C-system in its present form was introduced in \cite{Csubsystems}. The type of the C-systems is constructively equivalent to the type of contextual categories defined by Cartmell in \cite{Cartmell0} and \cite{Cartmell1} but the definition of a C-system is slightly different from the Cartmell's foundational definition.

In this paper we consider what might be the most important class of structures on C-systems - the structures that correspond, for syntactic C-systems, to the operations of dependent product,  $\lambda$-abstraction and application that satisfy the $\beta$ and $\eta$ rules. The first such structure was defined for general C-systems by John Cartmell in \cite[pp. 3.37 and 3.41]{Cartmell0} as a part of what he called a strong M.L. structure. It was later studied by Thomas Streicher in \cite[p.71]{Streicher} who called a C-system (contextual category) together with such a structure a ``contextual category with products of families of types''. 

The goal of this paper is to define another structure on C-systems, which we call the $(\Pi,\lambda)$-structure, and to establish a bijection between the set of Cartmell-Streicher structures and $(\Pi,\lambda)$-structures. 
The $(\Pi,\lambda)$-structures will be studied in \cite{Pilambda}. 

This paper, together with \cite{Pilambda}, forms a more detailed and systematic version of the earlier preprint \cite{fromunivwithPi}.  

A note must be made about our use of the expressions ``a structure'' and ``the structure''. In the latter case, as for example in ``the group structure'', we usually refer to a type of structure on some objects. If we ignore the small variations in the definition, there is only one notion of group structure and ``the group structure'' refers to this notion. On the other hand when we say ``a group structure on X'' we mean a particular instance or an element of the set of group structures. Thus we can talk about the Cartmell-Streicher structure and the $(\Pi,\lambda)$-structure on C-systems and also about the bijection between the set of Cartmell-Streicher structures and $(\Pi,\lambda)$-structures on a C-system. 

We start the paper with Section \ref{Sec.2.1} where we establish a number of general results about C-systems. Some of these results are new. Some have been stated by Cartmell \cite{Cartmell0} and Streicher \cite{Streicher}, but without proper mathematical proofs. Among notable new facts we can mention Lemma \ref{2016.06.11.l1} that shows that the canonical direct product in a C-system is strictly associative. 

In Section \ref{Sec.2.2} we construct on any C-system two families of presheaves - $\Ob_n$ and $\wtOb_n$. These presheaves play a major role in our approach to the C-system formulation of systems of operations that correspond to systems of inference rules. The main result here is Construction \ref{2016.06.17.constr1} for Problem \ref{2016.06.11.prob2}.
It is likely that constructions for various other variants of this problem involving morphisms between presheaves $\Ob_*$ and $\wtOb_*$ can be given. The full generality of this result should involve as the source fiber products of $\Ob_*$ and $\wtOb_*$ relative to morphisms satisfying certain properties and as the target $\Ob_*$ or $\wtOb_*$. We limit ourselves to Construction \ref{2016.06.17.constr1} here because it is the only case that will be required later in the paper. 

Up to Section \ref{Sec.2.1} all our results are about objects and morphisms of a single C-systems or about their behavior under homomorphisms of C-systems. Starting with Section \ref{Sec.2.2} we begin to consider presheaves on C-systems. There is a foundational issue related to the notion of a presheaf that is rarely if ever addressed. We discuss it in some detail in Remarks \ref{2016.09.05.rem1} and \ref{2016.11.08.rem1}. 

In Section \ref{Sec.2.3} we first recall the definition of the Cartmell-Streicher structure on a C-system. Then, in Definition \ref{2015.03.09.def2}, we give the main definition of the paper, the definition of a $(\Pi,\lambda)$-structure. In the rest of this section we work on constructing a bijection between the sets of Cartmell-Streicher structures and $(\Pi,\lambda)$-structures on a given C-system.

This bijection is the main result of the paper. Its construction uses most of the results of Section \ref{Sec.2.1} as well as results from the appendices. 

A Cartmell-Streicher structure on $CC$ can be seen as a pair $(\PPi,Ap)$ where $\PPi$ is a function $Ob_{\ge 2}\sr Ob$ satisfying  conditions of Definition \ref{2015.03.17.def1}(1) and $Ap$ is a function $Ob_{\ge 2}\sr Mor$ satisfying conditions of Definition \ref{2015.03.17.def1}(2) relative to $\PPi$.

A $(\Pi,\lambda)$-structure is a pair $(\Pi,\lambda)$ where $\Pi$ is a morphism of presheaves $\Ob_2\sr \Ob_1$ and $\lambda$ is a morphism of presheaves $\wtOb_2\sr \wtOb_1$ such that 
\begin{eq}
\llabel{2016.10.03.eq1}
\begin{CD}
\wtOb_2 @>\lambda>> \Ob_1\\
@V\partial VV @VV\partial V\\
\wtOb_1 @>\Pi>> \Ob_1
\end{CD}
\end{eq}
is a pullback. 

Substituting $i=2$ and $j=1$ in Construction \ref{2016.06.17.constr1} we obtain a bijection $\Phi$ from the set of morphisms of presheaves of the form $\Pi:\Ob_2\sr \Ob_1$ to the set of functions $\PPi:Ob_{\ge 2}\sr Ob$ satisfying the conditions of Definition \ref{2015.03.17.def1}(1). 

Let $All\lambda_1^{\Pi}$ be the set of morphisms $\lambda:\wtOb_2\sr \wtOb_1$ that make (\ref{2015.03.09.eq1}) a pullback, that is, which form, together with $\Pi$, a $(\Pi,\lambda)$-structure.  

Let $AllAp_1^\PPi$ be the set of functions $Ap:Ob_{\ge 2}\sr Mor$ that satisfy the conditions of Definition \ref{2015.03.17.def1}(2) relative to $\PPi$, that is, which form, together with $\PPi$, a Cartmell-Streicher structure. 

It remains to construct, for any morphism of presheaves $\Pi:\Ob_2\sr \Ob_1$,  a bijection of the form $All\lambda_1^{\Pi}\sr AllAp_1^{\Phi(\Pi)}$. 

The bijection that we construct is the composition of three bijections
\begin{eq}\llabel{2016.10.03.eq2}
All\lambda_1^{\Pi}\sr All\lambda_2^{\Pi}\sr AllAp_2^{\Phi(\Pi)}\sr AllAp_1^{\Phi(\Pi)}
\end{eq}

In this sequence the set $All\lambda_2^{\Pi}$ is the set of double families (families with two parameters) of bijections of the form
$$\partial^{-1}(B)\sr \partial^{-1}(\Pi_{\Gamma}(B))$$
parametrized by $\Gamma\in Ob$ and $B\in \Ob_2(\Gamma)$ that satisfy some naturality condition. The first bijection in (\ref{2016.10.03.eq2}), defined in Construction \ref{2016.10.03.constr1}, is a particular case of a bijection between the set of morphisms of presheaves on $\cal C$ that complete a given diagram of presheaves of the form 
$$
\begin{CD}
\wt{F} @. \wt{G}\\
@VaVV @VVbV\\
F @>P>> G
\end{CD}
$$
to a pullback square and the set of double families of bijections of the form
$$a_X^{-1}(A)\sr b_X^{-1}(P_X(A))$$
parametrized by $X\in {\cal C}$ and $A\in F(X)$ that satisfy some naturality condition. The general case is considered in Appendix B.

For $B\in Ob_{\ge 2}$ let $A=ft(B)$ and $\Gamma=ft^2(B)$. The set $AllAp_2^\PPi$ is defined in a very similar way to the set $AllAp_1^\PPi$ with the main difference that while $AllAp_1^{\PPi}$ is the set of families morphisms of the form 
$$\PPi(B)\times_{\Gamma}A\sr B$$
parametrized by $B\in Ob_{\ge 2}$ and satisfying certain conditions, the set $AllAp_2^{\PPi}$ is the set of families morphisms of the form 
$$A\times_{\Gamma}\PPi(B)\sr B$$
also parametrized by $B\in Ob_{\ge 2}$ and satisfying a somewhat different set of conditions. 

The bijection between the sets $All\lambda_2^{\Pi}$ and $AllAp_2^{\Phi(\Pi)}$ is, in a sense, the main one of the three bijections. It is defined by constructing two functions,
$$C1:All\lambda_2^{\Pi}\sr AllAp_2^{\Phi(\Pi)}$$
in Construction \ref{2015.03.13.constr1} and 
$$C2:AllAp_2^{\Phi(\Pi)}\sr All\lambda_2^{\Pi}$$
in Construction \ref{2015.03.15.constr1} and proving in Lemmas \ref{2015.03.15.l1} and \ref{2015.03.15.l2} that these functions are mutually inverse bijections. 

The last of the three bijections, the bijection between $AllAp_2^{\PPi}$ and $AllAp_1^{\PPi}$ is defined in Construction \ref{2016.08.12.constr1}. It and its inverse are given by the composition with the exchange morphisms 
$$exch(A,\PPi(B);\Gamma):A\times_{\Gamma}\PPi(B)\sr \PPi(B)\times_{\Gamma}A$$
and
$$exch(\PPi(B),A;\Gamma):\PPi(B)\times_{\Gamma}A\sr A\times_{\Gamma}\PPi(B)$$
that are defined and whose properties are proved in Section \ref{Sec.2.1}.

The $(\Pi,\lambda)$-structures correspond to the $(\Pi,\lambda,app,\beta,\eta)$-system of inference rules. In Remark \ref{2016.09.25.rem1} we outline the definitions of structures that correspond to the similar systems but without the $\beta$- or $\eta$-rules. Such structures appear as natural variations of the $(\Pi,\lambda)$-structures. 

Our main construction proceeds through two intermediate structures whose sets are denoted by $All\lambda_2^{\Pi}$ and $AllAp_2^{\PPi}$. This shows that there are other structures on C-systems that are equivalent to the Cartmell-Streicher and $(\Pi,\lambda)$-structures.

Among such structures there is an important one that is obtained by reformulating for C-systems the structure that is defined in \cite[Def. 5]{ClairDybjer} and that we may call the Clairambault-Dybjer structure. The C-system version of this structure is closer to the $(\Pi,\lambda)$-structure than to the Cartmell-Streicher structure and it should not be difficult to construct a bijection between Clairambault-Dybjer structures and $(\Pi,\lambda)$-structures. We leave this for a future paper.  

The methods of this paper are fully constructive. It is written in the formalization-ready style, that is, in such a way that no long arguments are hidden even when they are required only to substantiate an assertion that may feel obvious to readers who are closely associated with a particular tradition of mathematical thought. 

In regard to the actual formalization we, firstly, make our arguments accessible to the formalization in the standard ZF - the Zermelo-Fraenkel theory. Secondly, we make them accessible to the formalization in the UniMath language (see \cite{UniMath2015}). It is the latter that allows us to claim that out methods are constructive. We do not consider the questions that arise in connection with the accessibility of our arguments to the formalization in various intuitionistic versions of the ZF (\cite{Friedman1977}, \cite{Aczel1978}). 

The main result of this paper is not a theorem but a construction and so are many of the intermediate results. Because of the importance of constructions for this paper we use a special pair of names Problem-Construction for the specification of the goal of a construction and the description of a particular solution.

In the case of a Theorem-Proof pair one usually refers (by name or number) to the theorem when  using the proof of this theorem. This is acceptable in the case of theorems because the future use of their proofs is such that only the fact that there is a proof but not the particulars of the proof matter. 

In the case of a Problem-Construction pair the content of the construction often matters in the future use. Because of this we have to refer to the construction and not to the problem and we assign in this paper numbers both to Problems and to Constructions. 

In this paper we continue to use the diagrammatic order of writing composition of morphisms, i.e., for $f:X\sr Y$ and $g:Y\sr Z$ the composition of $f$ and $g$ is denoted by $f\circ g$.

For a functor $\Phi:C\sr C'$ we let $\Phi^{\circ}$ denote the functor $PreShv(C')\sr PreShv(C)$ given by pre-composition with a functor $\Phi^{op}:C^{op}\sr (C')^{op}$. On objects one has 
$$\Phi^{\circ}(F)(X)=F(\Phi(X))$$
In the literature this functor is denoted both by $\Phi^*$ and $\Phi_*$ and we decided to use a new unambiguous notation instead. 

Acknowledgements are at the end of the paper.


\subsection{General results on C-systems}
\label{Sec.2.1}

Some of the lemmas and theorems proved in this section can also be found in \cite{Cartmell1} and in \cite{Streicher}. However, many new results are included and we chose to provide independent proofs for a few known results for the convenience of the reference further in this paper and in the other papers of this series. 

Let us start by making some additions to the notations that were introduced in \cite{Csubsystems}. The new notations that we introduce are consistent with the notations introduced in \cite[pp.239-240]{Cartmell1}.
\begin{definition}\llabel{2016.08.06.def1}
Let $CC$ be a C-system. We will say that an object $X$ is over an object $Y$ and write $X\ge Y$ if $l(X)\ge l(Y)$ and $Y=ft^{l(X)-l(Y)}(X)$. We say that $X$ is above $Y$ and write $X>Y$ if $X$ is over $Y$ and $l(X)>l(Y)$. 
\end{definition}

Note that ``is over'' and ``is above'' are well-defined relations on $Ob(CC)$ with ``is over'' being reflexive and transitive and ``is above'' being transitive.  In addition one has
\begin{eq}
\llabel{2016.09.17.eq4}
{\rm if}\,\,X>\Gamma\,\,{\rm then}\,\,ft(X)\ge \Gamma
\end{eq} 
The following lemma provides an induction principle that in most proofs can be used instead of induction by length and that is more convenient than such induction.
\begin{lemma}
\llabel{2016.09.17.l1}
Let $\Gamma\in CC$ and let $P$ be a subset in $\{X\,|\,X\ge \Gamma\}$ such that 
\begin{enumerate}
\item $\Gamma\in P$,
\item if $X>\Gamma$ and $ft(X)\in P$ then $X\in P$.
\end{enumerate}
Then for all $X\ge \Gamma$, $X\in P$.
\end{lemma}
\begin{proof}
Let $X\ge \Gamma$ and $n=l(X)-l(\Gamma)$. Proceed by induction on $n$. For $n=0$ we have $X=\Gamma$ and therefore $X\in P$ by the first assumption. For the successor of $n$ we have that if $l(X)-l(\Gamma)=n+1$ then $X>\Gamma$. Therefore, by (\ref{2016.09.17.eq4}) we have $ft(X)\ge \Gamma$ and since $l(ft(X))-l(\Gamma)=n$ we have that $ft(X)\in P$ by the inductive assumption. We conclude that $X\in P$ by the second assumption of the lemma.
\end{proof}
\begin{remark}\rm
\llabel{2016.05.02.rem1b}
There is also another induction principle that can be used everywhere this one is used but also for purposes where this one fails. 

Since the notation ``ft'' comes from the word ``father'' we will call the concept that we want to introduce ``child''. For $X >Y$ in $CC$ denote by $ch(Y,X)$ and call "the child of $Y$ in the direction of $X$", the object $ft^{l(X)-l(Y)-1}(X)$. Then $X\ge ch(Y,X)>Y$ and $l(X)-l(ch(Y,X))=(l(X)-l(Y))-1$. There is a dual induction principle to the one that we stated above that uses the pairs $(X, ch(Y,X))$ instead of $(ft(X),Y)$.

Let, more generally, $ch_i(Y,X)=ft^{l(X)-l(Y)-i}(X)$. The advantage of using $ch(-,-)$ instead of $ft$ is that $ch(Y,X)$ are defined even in the systems where $X$ can be infinite over $Y$. 

Here we have to make a reference to the syntactic C-systems of type theories where $Ob(CC)$ is the set of contexts of the type theory (modulo alpha equivalence and possibly further equivalences). In formalization systems based on the univalent approach, for example in UniMath, structures such as $(\infty,1)$-categories or $A_{\infty}$-types are, intuitively, represented by infinite contexts. For example, the information about an $(\infty,1)$-category $\cal C$ is given by a type $Ob$, the morphisms family $Mor$, the family of composition functions and the family of the identity morphisms followed by an infinite sequence of families of equalities representing the higher associativity and identity axioms. For such an object $\cal C$, we have finite contexts $ch_i(pt,{\cal C})$ but not $ft({\cal C})$. 

In general, for every C-system $CC$ there is a category $\wh{CC}$ whose objects are objects of $CC$ together with extra objects, the set of which we can denote by $\wh{CC}_{\infty}$, which are infinite sequences $X_1,\dots,X_n,\dots$ where $X_i\in CC$ and $X_i=ft(X_{i+1})$. We can define morphisms between objects of $\wh{CC}$ using the usual definition of morphisms between pro-objects. 

Since their objects are connected with with structures that involve infinite sequences of ``coherence'' conditions as well as with certain kinds of co-inductive types categories $\wh{CC}$ deserve further study. 
\end{remark}
If $X\ge Y$ we will write $p(X,Y)$ for the composition of the p-morphisms going from $X$ to $Y$ that was previously denoted $p_{X,n}$ where $n=l(X)-l(Y)$. It follows immediately from its definition that 
\begin{eq}
\llabel{2016.09.17.eq1}
p(X,X)=Id_X\,\,{\rm and}\,\, p(X,Y)=p_X\circ p(ft(X),Y)\,\, {\rm for}\,\, X>Y
\end{eq}
If $X\ge Y$ and $Y\ge \Gamma$ then one has
\begin{eq}
\llabel{2015.06.11.eq4}
p(X,\Gamma)=p(X,Y)\circ p(Y,\Gamma)
\end{eq}
This is proved with Lemma \ref{2016.09.17.l1} by fixing $\Gamma$ and $Y$ and setting $P$ to be the set of $X\ge Y$ for which (\ref{2015.06.11.eq4}) holds. The assumptions of the lemma follow from (\ref{2016.09.17.eq1}). 

If $X\ge \Gamma$ and $f:\Gamma'\sr \Gamma$ is a morphism we will write $f^*(X)$ for what was previously denoted $f^*(X,n)$ where $n=l(X)-l(\Gamma)$ and
$$q(f,X):f^*(X)\sr X$$
for what was previously denoted by $q(f,X,n)$. It follows immediately from the definitions that
\begin{eq}
\llabel{2016.08.18.eq5}
f^*(\Gamma)=\Gamma'\,\,{\rm and}\,\,f^*(X)=q(f,ft(X))^*(X)\,\, {\rm for}\,\, X>\Gamma
\end{eq}
and
\begin{eq}
\llabel{2016.09.17.eq2}
q(f,\Gamma)=f\,\,{\rm and}\,\,q(f,X)=q(q(f,ft(X)),X)\,\, {\rm for}\,\, X>\Gamma
\end{eq}
The second half of (\ref{2016.08.18.eq5}) implies that for $X>\Gamma$ one has
\begin{eq}
\llabel{2016.05.06.eq1}
ft(f^*(X))=f^*(ft(X))
\end{eq}
\begin{lemma}
\llabel{2015.06.15.l1}
For any $X$ and $f$ as above $f^*(X)$ is an object over $\Gamma'$,
\begin{eq}\llabel{2016.04.30.eq2}
l(f^*(X))-l(\Gamma')=l(X)-l(\Gamma)
\end{eq}
and
\begin{eq}\llabel{2015.06.11.sq2}
\begin{CD}
f^*(X) @>q(f,X)>> X\\
@Vp(f^*(X),\Gamma')VV @VVp(X,\Gamma)V\\
\Gamma' @>f>> \Gamma
\end{CD}
\end{eq}
is a pullback.
\end{lemma}
\begin{proof}
Each of the three assertions is proved easily using Lemma \ref{2016.09.17.l1}. In the case of the third assertion one has to apply the facts that the canonical squares of a C-system are pullbacks and that the vertical composition of two pullbacks is a pullback. 
\end{proof}
A detailed definition of a homomorphism of C-systems is given in \cite[Definition 3.1]{Cfromauniverse}. 
\begin{lemma}
\llabel{2016.09.11.l2}
Let $H:CC'\sr CC$ be a homomorphism of C-systems. Then:
\begin{enumerate}
\item For $X\ge \Gamma$ in $CC'$ one has $H(X)\ge H(\Gamma)$ and 
$$H(p(X,\Gamma))=p(H(X),H(\Gamma))$$
\item For $X\ge \Gamma$ and $f:\Gamma'\sr\Gamma$ in $CC'$ one has
$$H(f^*(X))=H(f)^*(H(X))$$
$$H(q(f,X))=q(H(f),H(X))$$
\end{enumerate}
\end{lemma}
\begin{proof}
The proofs of all three assertions are through Lemma \ref{2016.09.17.l1} using the fact that homomorphisms of C-systems take $p$-morphisms to $p$-morphisms, respect $f^*$ on objects and take $q$-morphisms to $q$-morphisms.
\end{proof}
\begin{lemma}
\llabel{2015.06.11.l2}
For all $\Gamma$ and all $X\ge \Gamma$ one has:
\begin{enumerate}
\item $Id_{\Gamma}^*(X)=X$ and $q(Id_{\Gamma},X)=Id_X$,
\item if $f:\Gamma'\sr \Gamma$, $g:\Gamma''\sr \Gamma'$ are two morphisms then 
$$(g\circ f)^*(X)=g^*(f^*(X))$$
and 
$$q(g\circ f, X)=q(g, f^*(X))\circ q(f,X)$$
\end{enumerate}
\end{lemma}
\begin{proof}
The proofs of all assertions are through Lemma \ref{2016.09.17.l1} using the axioms of a C-system.
\end{proof}
\begin{lemma}
\llabel{2016.09.17.l2}
If $X\ge Y\ge \Gamma$ and $f:\Gamma'\sr \Gamma$ then one has
\begin{eq}
\llabel{2015.06.11.eq5}
f^*(X)=q(f,Y)^*(X)
\end{eq}
and
\begin{eq}
\llabel{2015.06.11.eq6}
q(f,X)=q(q(f,Y),X)
\end{eq}
\end{lemma}
\begin{proof}
One proves both statements simultaneously through Lemma \ref{2016.09.17.l1}. One fixes $\Gamma$ and $Y$ and sets $P$ to be the set of $X\ge Y$ such that (\ref{2015.06.11.eq5}) and (\ref{2015.06.11.eq6}) hold for $X$. One has $Y\in P$ by the first halves of (\ref{2016.08.18.eq5}) and (\ref{2016.09.17.eq2}). If $X>Y$ and $ft(X)\in P$ then $X>\Gamma$ and 
$$f^*(X)=q(f,ft(X))^*(X)=q(q(f,Y),ft(X))^*(X)=q(f,Y)^*(X)$$
where the first and the third equalities are by the second half of (\ref{2016.08.18.eq5}) and the second equality is by (\ref{2015.06.11.eq6}) for $ft(X)$. 

Similarly 
$$q(f,X)=q(q(f,ft(X)),X)=q(q(q(f,Y),ft(X)),X)=q(q(f,Y),X)$$
where the first and the third equalities are by the second half of (\ref{2016.09.17.eq2}) and the second equality is by (\ref{2015.06.11.eq6}) for $ft(X)$. 

This proves the second assumption of Lemma \ref{2016.09.17.l1} and completes the proof of our lemma.
\end{proof}
\begin{remark}\rm\llabel{2016.08.18.rem1}
Equations (\ref{2015.06.11.eq5}) and (\ref{2015.06.11.eq6}) are the subject of \cite[Lemma 14.1, p.240]{Cartmell1} and \cite[Lemma 1, p.2.14]{Cartmell0}. Some other constructions and lemmas of our text are used as given in the following few paragraphs of \cite{Cartmell1}. A few more results are stated and proved in \cite{Cartmell0}, which is unfortunately not published at this time. 
\end{remark}
The first assertion of Lemma \ref{2015.06.15.l1} together with (\ref{2015.06.11.eq5}) implies that if $X\ge Y\ge \Gamma$ and $f:\Gamma'\sr \Gamma$ then 
\begin{eq}
\llabel{2016.08.18.eq4}
f^*(X)\ge f^*(Y)
\end{eq}
\begin{lemma}
\llabel{2016.09.21.l3}
If $X\ge Y\ge \Gamma$ and $f:\Gamma'\sr \Gamma$ then the square
\begin{eq}
\llabel{2016.09.21.eq2}
\begin{CD}
f^*(X) @>q(f,X)>> X\\
@Vp(f^*(X),f^*(Y))VV @VVp(X,Y)V\\
f^*(Y) @>q(f,Y)>> Y
\end{CD}
\end{eq}
where the left vertical arrow is defined by (\ref{2016.08.18.eq4}), is a pullback.
\end{lemma}
\begin{proof}
By Lemma \ref{2016.09.17.l2} we have $f^*(X)=q(f,Y)^*(X)$ and $q(f,X)=q(q(f,Y),X)$. Therefore, our square coincides with the square of Lemma \ref{2015.06.15.l1} and is a pullback according to this lemma.
\end{proof}

If $X$ and $Y$ are objects over $\Gamma$ let
$$X\times_{\Gamma}Y=p(X,\Gamma)^*(Y)$$
Lemma \ref{2015.06.15.l1} shows that $X\times_{\Gamma} Y$ is the fiber product of $X$ and $Y$ over $\Gamma$ with the projections $p(X\times Y,X)$ and $q(p(X,\Gamma),Y)$.

The same lemma shows that
\begin{eq}\llabel{2016.06.17.eq3}
l(X\times_{\Gamma} Y)=l(X)+l(Y)-l(\Gamma)
\end{eq}
The product $X\times_{\Gamma} Y$ is an object over $X$ and therefore an object over $\Gamma$:
$$X\times_{\Gamma} Y \ge X\ge \Gamma$$
Note that $X\times_{\Gamma} Y$ is not, in general, an object over $Y$. 

We have two pullbacks
\begin{eq}
\label{2016.09.15.eq1}
\begin{CD}
X\times_{\Gamma} Y @>q(p(X,\Gamma),Y)>> Y\\
@Vp(p(X,\Gamma)^*(Y),X) VV @VVp(Y,\Gamma) V\\
X @>p(X,\Gamma)>> \Gamma
\end{CD}
\spc\spc\spc\spc\spc\spc
\begin{CD}
Y\times_{\Gamma} X @>q(p(Y,\Gamma),X)>> X\\
@Vp(p(Y,\Gamma)^*(X),Y) VV @VVp(X,\Gamma) V\\
Y @>p(Y,\Gamma)>> \Gamma
\end{CD}
\end{eq}
Applying Lemma \ref{2016.05.18.l2} and the construction preceding it to these squares we obtain an  isomorphism 
\begin{eq}\llabel{2016.05.18.eq3}
exch(X,Y;\Gamma):X\times_{\Gamma} Y\sr Y\times_{\Gamma} X
\end{eq}
with the inverse given by $exch(Y,X;\Gamma)$, that is,
\begin{eq}\llabel{2016.05.20.eq3}
\begin{CD}
exch(X,Y;\Gamma)\circ exch(Y,X;\Gamma)=Id_{X\times_{\Gamma} Y}\\
exch(Y,X;\Gamma)\circ exch(X,Y;\Gamma)=Id_{Y\times_{\Gamma} X}
\end{CD}
\end{eq}
This isomorphism is uniquely determined by two equalities
\begin{eq}\llabel{2016.06.18.eq6}
\begin{CD}
exch(X,Y;\Gamma)\circ q(p(Y,\Gamma),X)=p(X\times_{\Gamma} Y,X)\\
exch(X,Y;\Gamma)\circ p(Y\times_{\Gamma} X,Y)=q(p(X,\Gamma),Y)
\end{CD}
\end{eq}

The equalities (\ref{2016.06.18.eq6}) imply in particular that one has 
\begin{eq}\llabel{2016.05.18.eq5}
\begin{CD}
exch(X,\Gamma;\Gamma)=Id_X\\
exch(\Gamma,Y;\Gamma)=Id_Y
\end{CD}
\end{eq}
\begin{definition}\llabel{2016.08.06.def2}
Let $CC$ be a C-system and $\Gamma\in CC$. A morphism $a:X\sr Y$ in $CC$ is called a morphism over $\Gamma$ if $X$ and $Y$ are objects over $\Gamma$ and 
$$a\circ p({Y,\Gamma})=p({X,\Gamma})$$
\end{definition}
\begin{lemma}\llabel{2016.08.14.l3}
One has:
\begin{enumerate}
\item If $X\ge \Gamma$ then $Id_X$ is a morphism over $\Gamma$.
\item If $f:X\sr Y$ and $g:Y\sr Z$ are morphisms over $\Gamma$ then $f\circ g:X\sr Z$ is a morphism over $\Gamma$,
\item if $X,Y\ge \Gamma$ then $exch(X,Y;\Gamma)$ is a morphism over $\Gamma$.
\end{enumerate}
\end{lemma}
\begin{proof}
The first and the second assertions are verified by straightforward calculation. To verify the third assertion we have
$$exch(X,Y;\Gamma)\circ p(Y\times_{\Gamma} X,\Gamma)=exch(X,Y;\Gamma)\circ p(Y\times_{\Gamma} X,Y)\circ p(Y,\Gamma)=$$$$q(p(X,\Gamma),Y)\circ p(Y,\Gamma)=p(X\times_{\Gamma} Y,X)\circ p(X,\Gamma)=p(X\times_{\Gamma} Y,\Gamma)$$
where the first equality is by (\ref{2015.06.11.eq4}), the second is by (\ref{2016.06.18.eq6}), the third by the commutativity of the second square in (\ref{2016.09.15.eq1}) and the again fourth by (\ref{2015.06.11.eq4}).
\end{proof}
\begin{lemma}\llabel{2016.08.18.l2}
If $a:X\sr Y$ is a morphism over $\Gamma$ and $\Gamma$ is an object over $\Gamma'$ then $a$ is a morphism over $\Gamma'$.
\end{lemma}
\begin{proof}
Straightforward using (\ref{2015.06.11.eq4}).
\end{proof}
\begin{lemma}\llabel{2016.08.06.l1}
If $X,Y$ are objects over $\Gamma$, $a:X\sr Y$ is a morphism over $\Gamma$ and $f:\Gamma'\sr \Gamma$ is a morphism then there exists a unique morphism
$$f^*(a):f^*(X)\sr f^*(Y)$$
over $\Gamma'$ such that 
\begin{eq}\llabel{2015.06.11.eq7}
f^*(a)\circ q(f,Y)=q(f,X)\circ a
\end{eq}
\end{lemma}
\begin{proof}
It follows from the fact the the square (\ref{2015.06.11.sq2}) is a pullback.
\end{proof}
\begin{lemma}
\llabel{2016.09.11.l3}
Let $H:CC'\sr CC$ be a homomorphism of C-systems. Then one has
\begin{enumerate}
\item if $\Gamma\in CC'$ and $a:X\sr Y$ is a morphism over $\Gamma$ then $H(a)$ is a morphism over $H(\Gamma)$,
\item if $f:\Gamma'\sr \Gamma$ and $a:X\sr Y$ is a morphism over $\Gamma$ then 
$$H(f^*(a))=H(f)^*(H(a))$$
where the right hand side is defined by the first part of the lemma.
\end{enumerate}
\end{lemma}
\begin{proof}
The first assertion follows from Lemma \ref{2016.09.11.l2}(1) and the fact that $(H_{Ob},H_{Mor})$ is a functor.

To prove the second assertion one needs to verify that $H(f^*(a))$ is a morphism over $H(\Gamma')$ and that it satisfies the defining property (\ref{2015.06.11.eq7}) of $H(f)^*(H(a))$. The first fact follows from the first part of the lemma, the second from Lemma \ref{2016.09.11.l2}(2) and the fact that $(H_{Ob},H_{Mor})$ is a functor.
\end{proof}
\begin{lemma}
\llabel{2016.11.15.l1}
Let $X,Y\ge Z\ge \Gamma$, $a:X\sr Y$ a morphism over $Z$ and $f:\Gamma'\sr \Gamma$ a morphism. Then $a$ is a morphism over $\Gamma$ and one has
$$f^*(a)=q(f,Z)^*(a)$$
\end{lemma}
\begin{proof}
We need to show that $q(f,Z)^*(a)$ is a morphism over $\Gamma'$ and that
the equality
$$q(f,Z)^*(a)\circ q(f,Y)=q(f,X)\circ a$$
By construction, $q(f,Z)^*(a)$ is a morphism over $f^*(Z)$. Since $f^*(Z)\ge \Gamma'$ it is a morphism over $\Gamma'$ by Lemma \ref{2016.08.18.l2}.

Next we have
$$q(f,Z)^*(a)\circ q(f,Y)=q(f,Z)^*(a)\circ q(q(f,Z),Y)=q(q(f,Z),X)\circ a=q(f,X)\circ a$$
where the first equality is by (\ref{2015.06.11.eq6}), the second by the definition of $q(f,Z)^*(a)$ and the third is again by (\ref{2015.06.11.eq6}). 

The lemma is proved.
\end{proof}
\begin{lemma}\llabel{2016.08.18.l1}
One has:
\begin{enumerate}
\item if $X\ge \Gamma$ then
\begin{eq}
\llabel{2016.04.25.eq3}
f^*(Id_X)=Id_{f^*(X)}
\end{eq}
where the left hand side is defined by Lemma \ref{2016.08.14.l3}.
\item if $a:X\sr Y$, $b:Y\sr Z$ are morphisms over $\Gamma$ then
\begin{eq}
\llabel{2015.06.11.eq2}
f^*(a\circ b)=f^*(a)\circ f^*(b)
\end{eq}
where the left hand side is defined by Lemma \ref{2016.08.14.l3}.
\end{enumerate}
\end{lemma}
\begin{proof}
In each case we need to verify that the right hand side of the equality is a morphism over $\Gamma'$, that it has the same domain and codomain as the left hand side and that it satisfies the equality of the form (\ref{2015.06.11.eq7}) that characterizes the left hand side. 

In the first case, that the right hand side is a morphism over $\Gamma'$ follows from Lemma \ref{2016.08.14.l3} while the properties of the domain and codomain and the equality (\ref{2015.06.11.eq7}) are straightforward.

In the second case, that the right hand side is a morphism over $\Gamma'$ again follows from Lemma \ref{2016.08.14.l3}, the properties of the domain and codomain are again straightforward; finally for the equality (\ref{2015.06.11.eq7}) we have
$$f^*(a)\circ f^*(b)\circ q(f,Z)=f^*(a)\circ q(f,Y)\circ b=q(f,X)\circ a\circ b$$
\end{proof}
\begin{lemma}
\llabel{2016.09.21.l2}
For $X,Y\ge \Gamma$ and $f:\Gamma'\sr \Gamma$ one has
\begin{eq}
\llabel{2016.09.21.eq1}
f^*(X\times_{\Gamma} Y)=f^*(X)\times_{\Gamma'}f^*(Y)
\end{eq}
and
\begin{eq}
\llabel{2016.08.18.eq1}
f^*(exch(X,Y;\Gamma))=exch(f^*(X),f^*(Y);\Gamma')
\end{eq}
where the left hand side is defined by Lemma \ref{2016.08.14.l3} and the right hand side by Lemma \ref{2015.06.15.l1}.
\end{lemma}
\begin{proof}
For (\ref{2016.09.21.eq1}) we have
$$f^*(X\times_{\Gamma} Y)=f^*(p(X,\Gamma)^*(Y))=q(f,X)^*(p(X,\Gamma)^*(Y))=(q(f,X)\circ p(X,\Gamma))^*(Y)=$$$$(p(f^*(X),\Gamma')\circ f)^*(Y)=p(f^*(X),\Gamma')^*(f^*(Y))=f^*(X)\times_{\Gamma'}f^*(Y)$$
where the first equality is by definition of $\times_{\Gamma}$, second equality is by (\ref{2015.06.11.eq6}), the third by Lemma \ref{2015.06.11.l2}, the fourth by the commutativity of (\ref{2015.06.11.sq2}) the fifth by Lemma \ref{2015.06.11.l2} and the sixth by the definition of $\times_{\Gamma'}$. 

To prove (\ref{2016.08.18.eq1}) we need to verify that the right hand side of the equality is a morphism over $\Gamma'$, that it has the same domain and codomain as the left hand side and that it satisfies the equality of the form (\ref{2015.06.11.eq7}) corresponding to the left hand side.

That the right hand side is a morphism over $\Gamma'$ follows from Lemma \ref{2016.08.14.l3}.

That domain of the left hand side equals the domain of the right hand side follows from (\ref{2016.09.21.eq1}). The identical reasoning with $X$ and $Y$ exchanged proves that the codomains of the left hand side and the right hand side coincide. 

To complete the proof we need to show that
\begin{eq}
\llabel{2016.09.23.eq1}
exch(f^*(X),f^*(Y);\Gamma')\circ q(f,Y\times_{\Gamma} X)=q(f,X\times_{\Gamma} Y)\circ exch(X,Y;\Gamma)
\end{eq}
Both sides of this equality have domain $f^*(X)\times_{\Gamma'} f^*(Y)$ and codomain $Y\times_{\Gamma} X$. The codomain is the fiber product with the projections $p(Y\times_{\Gamma} X,Y)$ and $q(p(Y,\Gamma),X)$. Therefore it is sufficient to show that the compositions of the left and the right hand sides with these morphisms coincide. We have
$$exch(f^*(X),f^*(Y);\Gamma')\circ q(f,Y\times_{\Gamma} X)\circ  p(Y\times_{\Gamma} X,Y)=$$$$
exch(f^*(X),f^*(Y);\Gamma')\circ p(f^*(Y\times_{\Gamma} X),f^*(Y))\circ q(f,Y)=$$$$
exch(f^*(X),f^*(Y);\Gamma')\circ p(f^*(Y)\times_{\Gamma'}f^*(X),f^*(Y))\circ q(f,Y)=$$$$
q(p(f^*(X),\Gamma'),f^*(Y))\circ q(f,Y)=q(p(f^*(X),\Gamma')\circ f,Y)$$
where the first equality is by the commutativity of (\ref{2016.09.21.eq2}), the second by (\ref{2016.09.21.eq1}), the third by (\ref{2016.06.18.eq6}), the fourth by Lemma \ref{2015.06.11.l2}. 

Next we have
$$q(f,X\times_{\Gamma} Y)\circ exch(X,Y;\Gamma)\circ  p(Y\times_{\Gamma} X,Y)=
q(f,X\times_{\Gamma} Y)\circ q(p(X,\Gamma),Y)=$$$$
q(f,p(X,\Gamma)^*(Y))\circ q(p(X,\Gamma),Y)=
q(q(f,X),p(X,\Gamma)^*(Y))\circ q(p(X,\Gamma),Y)=$$$$
q(q(f,X)\circ p(X,\Gamma),Y)$$

where the first equality is by (\ref{2016.06.18.eq6}), the second by the definition of $\times_{\Gamma}$, the third by (\ref{2015.06.11.eq6}) and the fourth by Lemma \ref{2015.06.11.l2}. We conclude that (\ref{2016.09.23.eq1}) holds by the commutativity of (\ref{2015.06.11.sq2}). 
\end{proof}
\begin{lemma}
\llabel{2016.05.20.l1}
Let $\Gamma', X, Y$ be objects over $\Gamma$ and $a:X\sr Y$ a morphism over $\Gamma$. Then one has
\begin{eq}\llabel{2016.05.20.eq1}
p(\Gamma',\Gamma)^*(a)\circ exch(\Gamma',Y;\Gamma)=exch(\Gamma',X;\Gamma)\circ q(a,p(Y,\Gamma)^*(\Gamma'))
\end{eq}
and
\begin{eq}\llabel{2016.05.20.eq2}
exch(X,\Gamma';\Gamma)\circ p(\Gamma',\Gamma)^*(a)=q(a,p(Y,\Gamma)^*(\Gamma'))\circ exch(Y,\Gamma';\Gamma)
\end{eq}
\end{lemma}
\begin{proof}
Let us prove (\ref{2016.05.20.eq1}). The domain of both sides of (\ref{2016.05.20.eq1}) is $p(\Gamma',\Gamma)^*(X)$ and the codomain is $p(Y,\Gamma)^*(\Gamma')$. By Lemma \ref{2015.06.15.l1} the codomain is a pullback with projections $p(p(Y,\Gamma)^*(\Gamma'),Y)$ and $q(p(Y,\Gamma),\Gamma')$. Therefore it is sufficient to show that the compositions of the left and right hand sides with each of the projections are equal. 

For the compositions with the first projection we have
$$p(\Gamma',\Gamma)^*(a)\circ exch(\Gamma',Y;\Gamma)\circ 
p(p(Y,\Gamma)^*(\Gamma'),Y)=$$$$p(\Gamma,\Gamma')^*(a)\circ q(p(\Gamma',\Gamma),Y)=q(p(\Gamma',\Gamma),X)\circ a$$
where the first equality is by (\ref{2016.06.18.eq6}) and the second one is by (\ref{2015.06.11.eq7}). Also we have
$$exch(\Gamma',X;\Gamma)\circ q(a,p(Y,\Gamma)^*(\Gamma')\circ p(p(Y,\Gamma)^*(\Gamma'),Y)=$$$$exch(\Gamma',X;\Gamma)\circ p(p(X,\Gamma)^*(\Gamma'),X)\circ a=q(p(\Gamma',\Gamma),X)\circ a$$
where the first equality is by the commutativity of squares of the form (\ref{2015.06.11.sq2}) and the second one by (\ref{2016.06.18.eq6}).

For the compositions with the second projections we have in the first case
$$p(\Gamma',\Gamma)^*(a)\circ exch(\Gamma',Y;\Gamma)\circ q(p(Y,\Gamma),\Gamma')=$$$$p(\Gamma',\Gamma)^*(a)\circ p(p(\Gamma',\Gamma)^*(Y),\Gamma')=p(p(\Gamma,\Gamma')^*(X),\Gamma')$$
where the first equality is by (\ref{2016.06.18.eq6}) and the second is by the fact that $p(\Gamma',\Gamma)^*(a)$ is a morphism over $\Gamma'$. In the second case we have
$$exch(\Gamma',X;\Gamma)\circ q(a,p(Y,\Gamma)^*(\Gamma')\circ q(p(Y,\Gamma),\Gamma')=exch(\Gamma',X;\Gamma)\circ q(a\circ p(Y,\Gamma),\Gamma')=$$$$exch(\Gamma',X;\Gamma)\circ q(p(X,\Gamma),\Gamma')=p(p(\Gamma',\Gamma)^*(X),\Gamma')$$
where the first equality is by Lemma \ref{2015.06.11.l2}(2), the second one is by the fact that $a$ is a morphism over $\Gamma$ and the third one is by (\ref{2016.06.18.eq6}).

The second equality (\ref{2016.05.20.eq2}) follows by taking the composition of (\ref{2016.05.20.eq1}) with $exch(X,\Gamma';\Gamma)$ on the left and on the right and using (\ref{2016.05.20.eq3}).

The lemma is proved. 
\end{proof}
\begin{lemma}
\llabel{2016.04.28.l1}
Let $a:X\sr Y$ be a morphism over $\Gamma$. Then one has:
\begin{enumerate}
\item $Id_{\Gamma}^*(a)=a$,
\item if $f:\Gamma'\sr \Gamma$, $g:\Gamma''\sr \Gamma'$ are two morphisms then 
$$(g\circ f)^*(a)=g^*(f^*(a))$$
\end{enumerate}
\end{lemma}
\begin{proof}
\begin{enumerate}
\item It is sufficient to show that
$$a\circ q(Id_{\Gamma},Y)=q(Id_{\Gamma},X)\circ a$$
which is straightforward in view of Lemma \ref{2015.06.11.l2}(1). 
\item It is sufficient to show that
$$g^*(f^*(a))\circ q(g\circ f,Y)=q(g\circ f,X)\circ a$$
We have:
$$g^*(f^*(a))\circ q(g\circ f,Y)=g^*(f^*(a))\circ q(g, f^*(X))\circ q(f,X)=$$$$q(g,f^*(Y))\circ f^*(a)\circ q(f,X)=q(g,f^*(Y))\circ q(f,Y)\circ a=q(g\circ f,Y)\circ a$$
where the first and the fourth equalities follow from Lemma \ref{2015.06.11.l2}(2) and the second and the third from (\ref{2015.06.11.eq7}).
\end{enumerate}
\end{proof}
\begin{lemma}[cf. {\cite[Lemma, p.2.17]{Cartmell0}}]
\llabel{2015.06.15.l2}
\llabel{2015.06.11.l1}
Let $X,Y$ be objects over $\Gamma$, $Z$ an object over $Y$ and $a:X\sr Y$ a morphism over $\Gamma$. Let further $f:\Gamma'\sr \Gamma$ be a morphism. Then one has:
\begin{enumerate}
\item $p(Z,Y)$ is a morphism over $\Gamma$ and one has
\begin{eq}
\llabel{2016.04.28.eq1}
f^*(p(Z,Y))=p(f^*(Z),f^*(Y))
\end{eq}
\item $a^*(Z)$ is an object over $\Gamma$ and one has
\begin{eq}\llabel{2015.06.15.eq1}
f^*(a^*(Z))=(f^*(a))^*(f^*(Z))
\end{eq}
\item $q(a,Z)$ is a morphism over $\Gamma$ and one has
\begin{eq}\llabel{2015.06.15.eq2}
f^*(q(a,Z))=q(f^*(a),f^*(Z))
\end{eq}
\end{enumerate}
\end{lemma}
\begin{remark}\rm
The assertions of the lemma can be expressed by the informal equation 
$$
f^*\left(
\begin{CD}
a^*(Z) @>q(a,Z)>> Z\\
@VVV @VVp(Z,Y)V\\
X @>a>> Y
\end{CD}
\spc\,\,\,\,\right)
=
\begin{CD}
(f^*(a))^*(f^*(Z)) @> q(f^*(a),f^*(Z))>> f^*(Z)\\
@VVV @VVp(f^*(Z),f^*(Y))V\\
f^*(X) @>f^*(a)>> f^*(Y)
\end{CD}\spc\spc
$$
\end{remark}
\begin{proof}
\begin{enumerate}
\item That $p(Z,Y)$ is a morphism over $\Gamma$ follows from (\ref{2015.06.11.eq4}). By definition, $f^*(p(Z,Y))$ is the unique morphism over $\Gamma'$ satisfying the equation
\begin{eq}\llabel{2016.04.30.eq1}
f^*(p(Z,Y))\circ q(f,Y)=q(f,Z)\circ p(Z,Y)
\end{eq}
By Lemma \ref{2015.06.15.l1}, objects $f^*(Y)$ and $f^*(Z)$ are over $\Gamma'$ and by (\ref{2016.08.18.eq4}) we have $f^*(Z)\ge f^*(Y)$. Therefore $p(f^*(Z),f^*(Y))$ is defined and by the previous statement it is a morphism over $\Gamma'$. Next we have
$$p(f^*(Z),f^*(Y))\circ q(f,Y)=p(q(f,Y)^*(Z),f^*(Y))\circ q(f,Y)=$$$$
q(q(f,Y),Z)\circ p(Z,Y)=q(f,Z)\circ p(Z,Y)$$
Where the first equality is by (\ref{2015.06.11.eq5}), the second is by the commutativity of (\ref{2015.06.11.sq2}) for $\Gamma'=f^*(Y)$ and $\Gamma=Y$ and the third is by (\ref{2015.06.11.eq6}).

We conclude that $p(f^*(Z),f^*(Y))$ also satisfies (\ref{2016.04.30.eq1}) and therefore (\ref{2016.04.28.eq1}) holds.
\item By Lemma \ref{2015.06.15.l1}, $a^*(Z)$ is an object over $X$ and since $X$ is an object over $\Gamma$, $a^*(Z)$ is an object over $\Gamma$. For the proof of (\ref{2015.06.15.eq1}) 
we have 
$$f^*(a^*(Z))=q(f,X)^*(a^*(Z))=(q(f,X)\circ a)^*(Z)=(f^*(a)\circ q(f,Y))^*(Z)=$$$$(f^*(a))^*(q(f,Y)^*(Z))=(f^*(a))^*(f^*(Z))$$
where the first equality is by (\ref{2015.06.11.eq5}), the second is by Lemma \ref{2016.04.28.l1}(2), the third is by (\ref{2015.06.11.eq7}), the fourth is by Lemma \ref{2016.04.28.l1}(2) and the fifth is by (\ref{2015.06.11.eq5}). 
\item Let us show first that $q(a,Z)$ is a morphism over $\Gamma$. We have
$$q(a,Z)\circ p(Z,\Gamma)=q(a,Z)\circ p(Z,Y)\circ p(Y,\Gamma)=p(a^*(Z),X)\circ a\circ p(Y,\Gamma)=$$$$p(a^*(Z),X)\circ p(X,\Gamma)=p(a^*(Z),\Gamma)$$
where the first equality is by (\ref{2015.06.11.eq4}), the second by the commutativity of the square (\ref{2015.06.11.sq2}), the third is by the assumption that $a$ is a morphism over $\Gamma$ and the fourth is by (\ref{2015.06.11.eq4}).

By (\ref{2015.06.15.eq1}) the morphisms on the left and the right hand side of (\ref{2015.06.15.eq2}) have the same domain. The codomain of the morphisms on both sides of (\ref{2015.06.15.eq2}) is $f^*(Z)$. By (\ref{2015.06.11.eq5}) we have $f^*(Z)=q(f,Y)^*(Z)$ and by Lemma \ref{2015.06.15.l1}
$q(f,Y)^*(Z)$ is a pullback with projections 
$$p(q(f,Y)^*(Z),f^*(Y))=p(f^*(Z),f^*(Y))$$
and
$$q(q(f,Y),Z)=q(f,Z)$$
where the equalities are by (\ref{2015.06.11.eq5}) and (\ref{2015.06.11.eq6}).

It is, therefore sufficient to verify that the compositions of the right and the left hand sides of (\ref{2015.06.15.eq2}) with each of the projections coincide. We have 
$$f^*(q(a,Z))\circ p(f^*(Z),f^*(Y))=f^*(q(a,Z))\circ f^*(p(Z,Y))=$$$$f^*(q(a,Z)\circ p(Z,Y))$$
where the first equality is by (\ref{2016.04.28.eq1}) and the second by (\ref{2015.06.11.eq2}). Next we have
$$q(f^*(a),f^*(Z))\circ p(f^*(Z),f^*(Y))=p((f^*(a))^*(f^*(Z)),f^*(X))\circ f^*(a)=$$$$p(f^*(a^*(Z)),f^*(X))\circ f^*(a)=f^*(p(a^*(Z),X))\circ f^*(a)=f^*(p(a^*(Z),X)\circ a)=$$$$f^*(q(a,Z)\circ p(Z,Y))$$
where the first equality is by the commutativity of square (\ref{2015.06.11.sq2}), the second is by (\ref{2015.06.15.eq1}), the third is by (\ref{2016.04.28.eq1}), the fourth is by (\ref{2015.06.11.eq2}) and the fifth is again by the commutativity of square (\ref{2015.06.11.sq2}). This shows that the compositions with the first projection coincide. 

For the compositions with the second projection we have:
$$f^*(q(a,Z))\circ q(f,Z)=q(f,a^*(Z))\circ q(a,Z)=q(q(f,X),a^*(Z))\circ q(a,Z)=$$$$q(q(f,X)\circ a,Z)$$
where the first equality is by (\ref{2015.06.11.eq7}), the second by (\ref{2015.06.11.eq6}) and the third by Lemma \ref{2015.06.11.l2}(2). Next we have
$$q(f^*(a),f^*(Z))\circ q(f,Z)=q(f^*(a),f^*(Z))\circ q(q(f,Y),Z)=q(f^*(a)\circ q(f,Y),Z)=$$$$q(q(f,X)\circ a,Z)$$
where the first equality is by (\ref{2015.06.11.eq6}), the second by Lemma \ref{2015.06.11.l2}(2) and the third one is by (\ref{2015.06.11.eq7}). This shows that the composition with the second projections coincide and completes the proof of (\ref{2015.06.15.eq2}).
\end{enumerate}
\end{proof}
Lemma \ref{2015.06.15.l2} can be used to show that the fiber product $X\times_{\Gamma} Y$ is strictly associative.
\begin{lemma}
\llabel{2016.06.11.l1}
For any $X,Y,Z$ over $\Gamma$ one has
$$X\times_{\Gamma} (Y\times_{\Gamma} Z)=(X\times_{\Gamma} Y)\times_{\Gamma} Z$$
\end{lemma}
\begin{proof}
We have
$$X\times_{\Gamma}(Y\times_{\Gamma} Z)=$$$$p(X,\Gamma)^*(p(Y,\Gamma)^*(Z))=(p(X,\Gamma)^*(p(Y,\Gamma)))^*(p(X,\Gamma)^*(Z))=$$$$p(p(X,\Gamma)^*(Y),p(X,\Gamma)^*(\Gamma))^*(p(X,\Gamma)^*(Z))=p(X\times_{\Gamma} Y,X)^*(p(X,\Gamma)^*(Z))=$$$$(p(X\times_{\Gamma} Y,X)\circ p(X,\Gamma))^*(Z)=p(X\times_{\Gamma} Y,\Gamma)^*(Z)=$$$$(X\times_{\Gamma} Y)\times_{\Gamma} Z$$
where the first equality is by definition, the second equality is by (\ref{2015.06.15.eq1}), the third is by (\ref{2016.04.28.eq1}), the fourth is by definitions, the fifth is by Lemma \ref{2015.06.11.l2}(2), the sixth is by (\ref{2015.06.11.eq4}) and the seventh is by definition. The lemma is proved.
\end{proof}
For completeness let us also show that $\Gamma$ is a strict two sided unit of $\times_{\Gamma}$.
\begin{lemma}
\label{2016.09.25.l1}
For any $X\ge \Gamma$ one has
$$X\times_{\Gamma}\Gamma=X\spc\spc \Gamma\times_\Gamma X=X$$
\end{lemma}
\begin{proof}
We have
$$X\times_{\Gamma}\Gamma=p(X,\Gamma)^*(\Gamma)=X$$
by (\ref{2016.08.18.eq5}) and
$$\Gamma\times_{\Gamma} X=p(\Gamma,\Gamma)^*(X)=Id_{\Gamma}^*(X)=X$$
by Lemma \ref{2015.06.11.l2}(1).
\end{proof}

In the case when $Z=pt$ we obtain a binary direct product on $CC$ that is strictly associative. Following the usual convention we write $X\times Y$ instead of $X\times_{pt} Y$. 

If $X\ge Y\ge \Gamma$ then $p(X,Y)$ is a morphism over $\Gamma$ by Lemma \ref{2015.06.15.l2}(1). Therefore if $ft(X)\ge \Gamma$ then $p_X=p(X,ft(X))$ is a morphism over $\Gamma$. In particular, if $X>\Gamma$ then $p_X$ is a morphism over $\Gamma$ and by Lemma \ref{2015.06.15.l2}(1) and (\ref{2016.05.06.eq1}) we have that for $f:\Gamma'\sr \Gamma$,
\begin{eq}\llabel{2016.05.14.eq1}
f^*(p_X)=p_{f^*(X)}
\end{eq} 

For a morphism $p:X\sr Y$ in a category let $sec(p)$ be the set of sections of $p$, that is,
$$sec(p)=\{s:Y\sr X\,|\,s\circ p=Id_X\}$$
\begin{lemma}
\llabel{2016.09.21.l1}
Let $X,Y,\Gamma\in CC$ then one has:
\begin{enumerate}
\item if $p:X\sr Y$ is a morphism over $\Gamma$ and $s\in sec(p)$ then $s$ is a morphism over $\Gamma$ and for $f:\Gamma'\sr \Gamma$ one has 
$$f^*(s)\in sec(f^*(p))$$
\item if $X\ge Y$ and $s\in sec(p(X,Y))$ then $s$ is a morphism over $\Gamma$ if and only if $Y\ge \Gamma$ and in this case for $f:\Gamma'\sr \Gamma$ one has
$$f^*(s)\in sec(p(f^*(X),f^*(Y)))$$
\end{enumerate}
\end{lemma}
\begin{proof}
If $p:X\sr Y$ is a morphism over $\Gamma$ then $X$ and $Y$ are objects over $\Gamma$. Since $s:Y\sr X$ we conclude that domain and codomain of $s$ are objects over $\Gamma$. Next we have
$$s\circ p(X,\Gamma)=s\circ p\circ p(Y,\Gamma)=Id_Y\circ p(Y,\Gamma)=p(Y,\Gamma)$$
which shows that $s$ is a morphism over $\Gamma$. For $f:\Gamma'\sr \Gamma$ we have $f^*(s):f^*(Y)\sr f^*(X)$ and, by Lemma \ref{2016.08.18.l1}, we have $f^*(s)\circ f^*(p)=f^*(s\circ p)=f^*(Id_Y)=Id_{f^*(Y)}$. 

To prove the second assertion note that if $Y\ge \Gamma$ then by the transitivity of $\ge$ we have $X\ge \Gamma$. Therefore, $p(X,Y)$ is a morphism over $\Gamma$ by Lemma \ref{2015.06.15.l2}(1) and $s$ is a morphism over $\Gamma$ by the first part of our lemma. In addition, by the first part of the lemma $f^*(s)\in sec(f^*(p(X,Y)))$ and $f^*(p(X,Y))=p(f^*(X),f^*(Y))$ by (\ref{2016.04.28.eq1}). 

On the other hand, if $s$ is a morphism over $\Gamma$ then $Y$ is an object over $\Gamma$ as the domain of $s$. This completes the proof of the lemma.
\end{proof}

Let us recall the notion of a C0-system from \cite[Definition 2.1]{Csubsystems}. It consists of structure on two sets of morphisms $Ob$ and $Mor$ comprising the length, $ft$, domain, codomain and identity functions together with operations of composition, $p$-morphisms and $q$-morphisms. These are required to satisfy all of the Cartmell's axioms except  the condition that the canonical squares are pullbacks. 

A C-system is defined as a C0-system together with the $s$-morphisms operation that is a function $a\mapsto s_a$ from the subset of morphisms $a:X\sr Y$ such that $l(Y)>0$ to all morphisms and that satisfies the following conditions where $ft(a)=a\circ p_{Y}$:
\begin{eq}
\llabel{2016.05.10.eq1}
s\in sec(p_{(ft(a))^*(Y)})
\end{eq}
\begin{eq}
\llabel{2016.05.10.eq3}
s_a\circ q(ft(a),Y)=a
\end{eq}
and if $b:ft(Y)\sr ft(Z)$, $l(Z)>0$ and $Y=b^*(Z)$ then
\begin{eq}
\llabel{2016.05.10.eq4}
s_a=s_{a\circ q(b,Z)}
\end{eq}
(see \cite[Definition 2.3]{Csubsystems}). We show in \cite[Proposition 2.4]{Csubsystems} that the canonical squares in a C-system are pullbacks and that if the canonical squares in a C0-system are pullbacks then there exists on it a unique $s$-morphisms operation satisfying the required conditions. 
\begin{remark}\rm\llabel{2016.08.18.rem2}
The construction of $s_a$ was also known to Cartmell, see \cite[p. 2.19]{Cartmell0}, who denoted these morphisms by $`a`$ and proved many interesting facts about them. However, his collection of results has very little intersection with the results that we use.
\end{remark}
\begin{lemma}\llabel{2016.05.07.l1}
Let $a:X\sr Y$ be a morphism over $\Gamma$ and $Y>\Gamma$. Then for $f:\Gamma'\sr \Gamma$ one has:
\begin{enumerate}
\item $ft(a)$ is a morphism over $\Gamma$ and one has
\begin{eq}\llabel{2016.05.10.eq5}
f^*(ft(a))=ft(f^*(a))
\end{eq}
\item $s_a$ is a morphism over $\Gamma$ and one has
\begin{eq}\llabel{2016.05.07.eq1}
f^*(s_a)=s_{f^*(a)}
\end{eq}
\end{enumerate}
\end{lemma}
\begin{proof}
Since $Y>\Gamma$ the morphism $p_Y$ is a morphism over $\Gamma$ and $ft(a)$ is a morphism over $\Gamma$ as the composition of two morphisms over $\Gamma$. 

Next we have
$$
f^*(ft(a))=f^*(a\circ p_Y)=f^*(a)\circ f^*(p_Y)=f^*(a)\circ p_{f^*(Y)}=ft(f^*(a))
$$
where the second equality is by (\ref{2015.06.11.eq2}) and the third equality is by (\ref{2016.05.14.eq1}). This proves (\ref{2016.05.10.eq5}). 

We have
$$p_{(ft(a)^*(Y)}=p(ft(a)^*(Y),X)$$
Therefore, the morphism $s_a$ is a morphism over $\Gamma$ by (\ref{2016.05.10.eq1}), Lemma \ref{2016.09.21.l1}(2) and our assumption that $X\ge \Gamma$. 

It remains to prove (\ref{2016.05.07.eq1}). The domains of two sides of (\ref{2016.05.07.eq1}) coincide. The codomain of the left hand side is $f^*(ft(a)^*(Y))$ and the codomain of the right hand side is $ft(f^*(a))^*(f^*(Y))$. We have
\begin{eq}\llabel{2016.05.12.eq1}f^*(ft(a)^*(Y))=(f^*(ft(a)))^*(f^*(Y))=ft(f^*(a))^*(f^*(Y))
\end{eq}
where the first equality is by (\ref{2015.06.15.eq1}) and the second one by (\ref{2016.05.10.eq5}). Therefore the codomains of the two sides of (\ref{2016.05.07.eq1}) coincide. 

Since the canonical squares in a C-system are pullbacks, $ft(f^*(a))^*(f^*(Y))$ is a pullback with projections $p_{ft(f^*(a))^*(f^*(Y))}$ and $q(ft(f^*(a)),f^*(Y))$. Therefore it is sufficient to verify that the compositions of the two sides of (\ref{2016.05.07.eq1}) with these projections coincide. 

We have
$$f^*(s_a)\circ p_{ft(f^*(a))^*(f^*(Y))}=f^*(s_a)\circ p_{f^*(ft(a)^*(Y))}=f^*(s_a)\circ f^*(p_{ft(a)^*(Y)})=$$$$f^*(s_a\circ p_{ft(a)^*(Y)})=f^*(Id_X)=Id_{f^*(X)}=s_{f^*(a)}\circ p_{ft(f^*(a))^*(f^*(Y))}$$
where the first equality is by (\ref{2016.05.12.eq1}), the second equality is by (\ref{2016.05.14.eq1}), the third equality is by (\ref{2015.06.11.eq2}), the fourth equality is by (\ref{2016.05.10.eq1}), the fifth equality is by (\ref{2016.04.25.eq3}) and the sixth is by (\ref{2016.05.10.eq1}). 

Next we have
$$f^*(s_a)\circ q(ft(f^*(a)),f^*(Y))=f^*(s_a)\circ q(f^*(ft(a)),f^*(Y))=f^*(s_a)\circ f^*(q(ft(a),Y))=$$$$f^*(s_a\circ q(ft(a),Y))=f^*(a)=s_{f^*(a)}\circ q(ft(f^*(a)),f^*(Y))$$
where the first equality is by (\ref{2016.05.10.eq5}), the second equality is by (\ref{2015.06.15.eq2}), the third by (\ref{2015.06.11.eq2}), the fourth by (\ref{2016.05.10.eq3}) and  the fifth again by (\ref{2016.05.10.eq3}).

This completes the proof of the lemma.
\end{proof}

\subsection{Presheaves $\Ob_n$ and $\wOb_n$}
\label{Sec.2.2}
Given a C-system $CC$ and an object $\Gamma$ of $CC$ we let 
\begin{eq}
\llabel{2016.09.01.eq1}
\Ob_n(\Gamma)=\{X\in CC\,|\,l(X)=l(\Gamma)+n,\,\,ft^n(X)=\Gamma\}
\end{eq}
In particular, for any $n$ and $X\in\Ob_n(\Gamma)$ we have $X\ge \Gamma$. Therefore, for $f:\Gamma'\sr \Gamma$ the object $f^*(X)$ is defined and by (\ref{2016.04.30.eq2}) we have $f^*(X)\in\Ob_n(\Gamma)$. By Lemma \ref{2015.06.11.l2}, the functions $f^*:\Ob_n(\Gamma)\sr\Ob_n(\Gamma')$ satisfy the axioms of a presheaf (see Remark \ref{2016.09.05.rem1}). We keep the notation $\Ob_n$ for this presheaf and may write $\Ob_n(f)$ for the function $f^*$ on $\Ob_n$.

We also let
\begin{eq}
\llabel{2016.09.01.eq2}
\wt{\Ob}_n(\Gamma)=\{o\in Mor(CC)\,|\,codom(o)\in\Ob_n(\Gamma),\,\,o\in sec(p_{codom(o)}),\,\,codom(o)>\Gamma\}
\end{eq}
The last condition is automatically satisfied if $n>0$ and implies that $\wOb_0(\Gamma)=\emptyset$. 

For $o:X\sr Y$ in $\wt{\Ob}_n(\Gamma)$ we have $dom(o)=ft(codom(o))$, and $codom(o)>\Gamma$. Therefore $dom(o)\ge \Gamma$ by (\ref{2016.09.17.eq4}) and $o$ is a morphism over $\Gamma$ by Lemma \ref{2016.09.21.l1}. 

By the previous conclusion and Lemma \ref{2016.08.06.l1}, for any $f:\Gamma'\sr \Gamma$ there is a unique morphism $f^*(o)$ of the form $f^*(X)\sr f^*(Y)$ over $\Gamma'$ such that $f^*(o)\circ q(f,Y)=q(f,X)\circ o$. 

By (\ref{2016.04.30.eq2}) we have $codom(f^*(o))\in\Ob_n(\Gamma')$. By (\ref{2016.04.28.eq1}) we have $p_{codom(f^*(o))}=f^*(p_{codom(o)})$ and therefore by Lemma \ref{2016.08.18.l1} we have $f^*(o)\in sec(p_{codom(f^*(o))})$. Finally since $n>0$ we have $codom(f^*(o))>\Gamma$. 

We conclude that for $o\in\wt{\Ob}_n(\Gamma)$ and $f:\Gamma'\sr\Gamma$ the morphism $f^*(o)$ is defined and belongs to $\wt{\Ob}_n(\Gamma')$. By Lemma \ref{2016.04.28.l1} the functions $f^*$ on $\wOb_n$ satisfy the axioms of a presheaf, that is, a contravariant functor from $CC$ to $Sets$. We keep the notation $\wt{\Ob}_n$ for this presheaf and may write $\wt{\Ob}_n(f)$ for the function $f^*$ on $\wt{\Ob}_n$.

For $o\in\wt{\Ob}_n(\Gamma)$ we let $\partial(o)$ denote $codom(o)$. It is immediate from the definitions that this defines morphisms of presheaves
$$\partial:\wt{\Ob}_n\sr\Ob_n$$
%


%
\begin{problem}\llabel{2016.06.11.prob2}
Let $i,j\in \nn$, $i>0$. To construct a bijection between the following two sets:
\begin{enumerate}
\item The set of functions $F:Ob_{\ge i}\sr Ob$ such that 
\begin{enumerate}
\item for any $X\in Ob_{\ge i}$ one has $ft^j(F(X))=ft^i(X)$,
\item for any $X\in Ob_{\ge i}$ and $f:\Gamma'\sr ft^i(X)$ one has $f^*(F(X))=F(f^*(X))$.
\end{enumerate}
\item The set of morphisms of presheaves $\Ob_i\sr \Ob_j$ on $CC$.
\end{enumerate}
\end{problem}
\begin{remark}\rm\llabel{2016.06.13.rem1}
For $i=0$ there may be functions $F$ that do not correspond to morphisms of presheaves $\Ob_0\sr \Ob_j$. Consider for example  the one point C-system $Pt$ such that $Ob(Pt)=\{pt\}$. Then for any $j$ there is a unique function $Ob_{\ge 0}\sr Ob$ satisfying conditions (a) and (b). On the other hand, $\Ob_0$ is the one point presheaf while $\Ob_j$ for $j>0$ is the empty presheaf and the set of morphisms $\Ob_0\sr \Ob_j$ is empty. However the result may remain valid if the condition that $F(X)\in Ob_{\ge j}$ is added. 
\end{remark}
We start with an intermediate construction and a lemma.
\begin{problem}
\llabel{2016.06.11.prob1}
Let $X\in Ob_{\ge 1}$. For any $Y$ and $i\in \nn$ to construct an object $X'$ such that $X'\ge Y$ and $l(X')=l(Y)+i$.
\end{problem}
\begin{construction}\rm\llabel{2016.06.11.constr1}
Let $X_1=ft^{l(X)-1}(X)$. Then $l(X_1)=1$. For any $Z$, define $Z^n$ inductively by the rule $Z^0=pt$, $Z^1=Z$ and $Z^{n+1}=Z^n\times Z$. By (\ref{2016.06.17.eq3}) we have 
$$l(Z^{n+1})=l(Z^n)+l(Z)$$
This implies that $l(Z^n)=n l(Z)$. In particular, $l(X_1^i)=i$. Then $l(Y\times X_1^i)=l(Y)+i$ and $Y\times X_1^i\ge Y$.
\end{construction}
\begin{lemma}
\llabel{2016.06.17.l1}
Let $j\in\nn$ and let $F:Ob_{\ge i}\sr Ob$ be a function satisfying the conditions (a),(b) of Problem \ref{2016.06.11.prob2} relative to $j$. Then for any $X\in Ob_{\ge i}$ one has 
$$l(F(X))=l(ft^i(X))+j$$
\end{lemma}
\begin{proof}
The issue that we have to address is that in the case $l(X)=i$ we have $l(ft^i(X))=l(ft^j(F(X)))=0$ which only tells us that $l(F(X))\le j$. We will show that this can not occur for functions satisfying the second condition of the problem. 

Let $X\in Ob_{\ge i}$. Since $i>0$ we have $l(X)>0$ and by Construction \ref{2016.06.11.constr1} we obtain an object $X'$ such that $X'\ge ft^i(X)$ and $l(X')=i>0$. Let $p=p(X',ft^i(X))$. Then 
\begin{eq}\llabel{2016.06.17.eq2}
l(F(X))-l(ft^i(X))=l(p^*(F(X)))-l(X')=l(F(p^*(X)))-l(X')
\end{eq}
where the first equation is by (\ref{2016.04.30.eq2}) and the second by our condition on $F$. We also have
$$ft^j(F(p^*(X)))=ft^i(p^*(X))=X'$$
and since $l(X')>0$ this implies that 
$$l(ft^j(F(p^*(X))))=l(F(p^*(X)))-j$$
Therefore, $l(F(p^*(X)))=l(X')+j$ and from (\ref{2016.06.17.eq2}) we get that 
$$l(F(X))=l(ft^i(X))+j$$
The lemma is proved.
\end{proof}
\begin{construction}\rm\llabel{2016.06.17.constr1}
Let us provide a construction for Problem \ref{2016.06.11.prob2}. 

Let $Fn$ be the set of functions satisfying conditions (a),(b) and $Mr$ the set of morphisms of presheaves $\Ob_i\sr \Ob_j$. 

An element of $Mr$ is, by definition, a family of functions $\psi_{\Gamma}:\Ob_i(\Gamma)\sr \Ob_j(\Gamma)$ parametrized by $\Gamma\in Ob$ that satisfy the condition that for any $f:\Gamma'\sr \Gamma$ and $X\in \Ob_i(\Gamma)$ one has 
\begin{eq}\llabel{2016.06.17.eq4}
f^*(\psi_{\Gamma}(X))=\psi_{\Gamma'}(f^*(X))
\end{eq}
Given an element $\psi_*=(\psi_{\Gamma})_{\Gamma\in Ob}$ in $Mr$ define a function $\Phi(\psi_*):Ob_{\ge i}\sr Ob$ by the formula
$$\Phi(\psi_*)(X)=\psi_{ft^i(X)}(X)$$
The right hand side is defined because the assumption that $l(X)\ge i$ implies that $X\in \Ob_i(ft^i(X))$. We have
$$ft^j(\Phi(\psi_*)(X))=ft^j(\psi_{ft^i(X)}(X))=ft^i(X)$$
because $\psi_{\Gamma}:\Ob_i(\Gamma)\sr \Ob_j(\Gamma)$. For a morphism $f:\Gamma'\sr ft^i(X)$ we have
$$f^*(\Phi(\psi_*)(X))=f^*(\psi_{ft^i(X)}(X))=\psi_{\Gamma}(f^*(X))$$
by (\ref{2016.06.17.eq4}). Therefore $\Phi(\psi_*)\in Fn$ and we have constructed a function $\Phi:Mr\sr Fn$.

Let $F\in Fn$. Define a family of functions $\Psi(F)_{\Gamma}:\Ob_i(\Gamma)\sr \Ob_j(\Gamma)$ parametrized by $\Gamma\in Ob$ by the formula
$$\Psi(F)_{\Gamma}(X)=F(X)$$
To show that this formula defines a function to $\Ob_j(\Gamma)$ we need to show that $F(X)\ge \Gamma$ and $l(F(X))=l(\Gamma)+j$. For $X\in \Ob_i(\Gamma)$ we have $\Gamma=ft^i(X)$ and therefore $\Gamma=ft^j(F(X))$ since $F\in Fn$. This shows that $F(X)\ge \Gamma$. The equality $l(F(X))=l(\Gamma)+j$ follows from $\Gamma=ft^i(X)$ and the equality of Lemma \ref{2016.06.17.l1}.

Let $f:\Gamma'\sr \Gamma$. Then 
$$f^*(\Psi(F)_{\Gamma}(X))=f^*(F(X))=F(f^*(X))=\Psi(F)_{\Gamma'}(f^*(X))$$
where the second equality follows from $\Gamma=ft^i(X)$  and the fact that $F$ satisfies condition (b). This shows that the family $\Psi(F)_*$ is a morphism of presheaves and that we have constructed a function $\Psi:Fn\sr Mr$.

For $\psi_*\in Mr$, $\Gamma\in Ob$ and $X\in \Ob_i(\Gamma)$ we have
$$\Psi(\Phi(\psi_*))_{\Gamma}(X)=\Phi(\psi_*)(X)=\psi_{ft^i(X)}(X)=\psi_{\Gamma}(X)$$
because $\Gamma=ft^i(X)$ for $X\in \Ob_i(\Gamma)$. This shows that $\Phi\circ\Psi=Id_{Mr}$.

For $F\in Fn$ and $X\in Ob_{\ge i}$ we have
$$\Phi(\Psi(F))(X)=\Psi(F)_{ft^i(X)}(X)=F(X)$$
This shows that $\Psi\circ \Phi=Id_{Fn}$ and completes the construction. 
\end{construction}

Let us make a few remarks concerning presheaves $\Ob_i$ and $\wOb_i$ and homomorphisms of C-systems.
\begin{lemma}
\llabel{2016.09.11.l1}
Let $H:CC\sr CC'$ be a homomorphism of C-systems. Then for $\Gamma\in CC$ one has:
\begin{enumerate}
\item for $T\in \Ob_n(\Gamma)$ one has $H(T)\in \Ob_n(H(\Gamma))$,
\item for $o\in \wOb_n(\Gamma)$ one has $H(o)\in \wOb_n(H(\Gamma))$.
\end{enumerate}
\end{lemma}
\begin{proof}
The first assertion follows immediately from the fact that $H_{Ob}$ commutes with $l$-functions and the $ft$-function.

The second assertion follows from the fact that $H_{Ob}$ commutes with the $l$-functions and the $ft$-functions, from the fact that $(H_{Ob},H_{Mor})$ is a functor and from the fact that $H_{Mor}$ commutes with the $p$-functions.
\end{proof}
\begin{lemma}
\llabel{2016.09.11.l4}
Let $H:CC\sr CC'$ be a homomorphism of C-systems. Then for $f:\Gamma'\sr \Gamma$ one has
\begin{enumerate}
\item for $T\in\Ob_n(\Gamma)$ one has $H(f^*(T))=H(f)^*(H(T))$,
\item for $o\in \wOb_n(\Gamma)$ one has $H(f^*(o))=H(f)^*(H(o))$,
\end{enumerate}
where the right hand side of the equalities are defined by Lemma \ref{2016.09.11.l1}.
\end{lemma}
\begin{proof}
The first assertion follows from Lemma \ref{2016.09.11.l2}(2). The second from Lemma \ref{2016.09.11.l3}.
\end{proof}
Lemma \ref{2016.09.11.l4}(1) shows that the family of functions 
$$H\Ob_{i,\Gamma}:\Ob_i(\Gamma)\sr \Ob_i(H(\Gamma))$$
given by $H\Ob_{i,\Gamma}(T)=H(T)$ and defined in view of Lemma \ref{2016.09.11.l1}(1) is a morphism of presheaves
\begin{eq}
\llabel{2016.09.11.eq5}
H\Ob_i:\Ob_i\sr H^{\circ}(\Ob_i)
\end{eq}
Lemma \ref{2016.09.11.l4}(2) shows that the family of functions 
$$H\wOb_{i,\Gamma}:\wOb_i(\Gamma)\sr \wOb_i(H(\Gamma))$$
given by $H\wOb_{i,\Gamma}(o)=H(o)$ and defined in view of Lemma \ref{2016.09.11.l1}(2) is a morphism of presheaves
\begin{eq}
\llabel{2016.09.11.eq6}
H\wOb_i:\Ob_i\sr H^{\circ}(\wOb_i)
\end{eq}

At the conclusion of this section let us raise the issue that one encounters when trying to formalize the reasoning about presheaves both in the Zermelo-Fraenkel theory (ZF) and in the univalent type-theoretic formalization systems such as UniMath. 

If a presheaf is defined as a contravariant functor to the category of sets then in reasoning about presheaves this category of sets must be specified, that is, the set $U=Ob(Sets)$ must be chosen. This creates an extra parameter in the theory and one must quantify over this parameter formulating all statements that mention presheaves as starting with ``for all 
$U$ such that ...  one has ...'' and then separately proving that such an $U$ exists. In some cases one must also include statements about the independence of the obtained result from $U$, that is, statements that start  with ``for all $U_1$, $U_2$ such that ... one has ...''. Clearly, this would make the text very hard to read. 

Remark \ref{2016.09.05.rem1} below may be considered as outlining the beginnings of a small theory that may be used to make the reasoning about presheaves without specifying $U$ fully precise with respect to the ZF. 

Remark \ref{2016.11.08.rem1} addresses the same issue for the univalent style type-theoretic formalization systems on the example of the UniMath. The theory that is needed in this case is different. 

This in itself may be a reason to consider these theories as lying outside of mathematics proper and in the interfaces between mathematics and formal foundations of mathematics.  

Since in this paper we concentrate on the mathematical part of the theory we will continue to talk about presheaves as is customary today without paying attention to the issues discussed in these comments. 
\begin{remark}\rm
\llabel{2016.09.05.rem1}
The arguments in this remark are specific to the Zermelo-Fraenkel theory. Very different arguments related to the same basic issue in the UniMath formalization are discussed in the next remark.

In set theory, any set $U$ defines a category of sets $Sets(U)$ whose set of objects is $U$ and morphisms $X\sr Y$ are functions from $X$ to $Y$, that is,
$$Mor(Sets(U))=\cup_{X,Y\in U} Fun(X,Y)$$
The required functions of domain, codomain etc. are defined because of our choice of the definition of what a function is, which we take to be the one given in \cite[p.81]{Bourbaki.Sets}, where a function $f$ is a triple $(G,X,Y)$ where $X=dom(f)$, $Y=codom(f)$ and $G\subset X\times Y$ is the graph of $f$. 

For most sets $U$ the category $Sets(U)$ will not have the properties usually expected from ``the'' category of sets. In order for $Sets(U)$ to have limits and colimits etc. the set $U$ should be closed under a number of constructions, i.e., to be a ``universe''. In order for our results concerning contravariant functors to $Sets(U)$ to hold one can choose $U$ whose existence can be proved in the ZF. In many papers one requires $U$ to be a Grothendieck universe (see \cite[I, Appendice par N. Bourbaki]{SGA4}), however existence of a Grothendieck universe is a very strong axiom that is not provable in the ZF. 

Instead of working with $U$ as a parameter one can define the notion of a presheaf of sets differently in a way that does not require a choice of any particular category of sets. 

Let us introduce the following terminology. As we said above a function is a triple $(G,X,Y)$ where $G$ is a functional graph such that $pr_1(G)=X$ and $pr_2(G)\subset Y$ (see {\em loc.cit.}). Define a {\em family} (of sets) parametrized by $X$ as a pair $(G,X)$ where $G$ is a functional graph such that $pr_1(G)=X$.

Any family can be extended to a function by choosing $Y=pr_2(G)$. However, choosing $Y$ to be any set that contains $pr_2(G)$ will work as well which shows that a family can be extended to a function in many ways. 

Given a family $A=(G,X)$ and $x\in X$ there exists a unique $y$ such that $(x,y)\in G$. We will write both $A_x$ and $A(x)$ for this $y$. Given two families $A$ and $F$ parametrized by $X$ we say that $F$ is a family of elements of $A_x$ if for all $x\in X$ one has $F(x)\subset A_x$. 

We can now define a presheaf of sets on $\cal C$ as a pair $(F_{Ob},F_{Mor})$ where:
\begin{enumerate}
\item $F_{Ob}$ is a family of sets parametrized by $Ob({\cal C})$,
\item $F_{Mor}$ is a family of functions of the form 
$$F_{Mor}(f):F_{Ob}(codom(f))\sr F_{Ob}(dom(f))$$
parametrized by $Mor({\cal C})$,
\end{enumerate}
and such that the usual axioms of a presheaf hold. We will call the objects just defined ``presheaves'' while the contravariant functors to $Sets(U)$, the $U$-presheaves. 

It is easy to define the notion of a morphism of presheaves, of the identity morphism and of composition of morphisms. 

The usual axioms of a category are satisfied for these definitions, but presheaves so defined do not form a set and therefore there is no ``category of presheaves''.  Instead what is available to us is a collection of notions - objects, morphisms, compositions etc. with the same logical structure as the one we have for elements of the sets with the corresponding names associated with a category. We will call such a collection a ``meta-category''. Until a precise meaning at the level of the formal first order logic is provided for the idea of the logical structure used above ``meta-category'' can not be considered a precise concept, but it can be used as a convenient verbal tool.

That presheaves do not form a category is a disadvantage of this definition. On the other hand a strong advantage of it is that it does not require an additional parameter $U$. 

As is usual, when no confusion is possible we will omit the indexes $Ob$ and $Mor$ at $F$ both for presheaves and for $U$-presheaves. 

If $F$ is a presheaf and $U$ is a set such that for all $X\in Ob({\cal C})$ one has $F_{Ob}(X)\in U$ then there is a unique $U$-presheaf $F_{U}$ such that for all $X\in Ob({\cal C})$ one has $F_U(X)=F(X)$. If $F_{Ob}(X)\in U$ and $G_{Ob}(X)\in U$ for all $X$ then 
$$Mor(F,G)=Mor_{PreShv({\cal C},U)}(F_U,G_U)$$
where $PreShv({\cal C},U)$ is the category of $U$-presheaves. 

If $F_{Ob}=(G,Ob({\cal C}))$ then taking $U=pr_2(G)$ we can define $F_U$ as above. This proves that for any presheaf there exists $U$ and a $U$-presheaf $F_U$ corresponding to $F$. The same can be shown, using the union axiom, for any family of presheaves parametrized by a set, that is, for any such family there exists a single $U$ such that all members of the family have the corresponding $U$-presheaves. 

In this paper we will be interested in two properties of diagrams of presheaves and their morphisms - when such a diagram is commutative and when, if this diagram is a square, it is a pullback.

Since presheaves do not form a category the latter requires a definition. 
Fortunately, the powerful tool of unbounded quantification is available to us in the ZF and using it we can directly transport the categorical definition of what a pullback is to the meta-category of presheaves. 

A criterion that only requires bounded quantification for when a commutative square of presheaves is a pullback is given in Lemma \ref{2016.08.02.l2}.

In addition one can show that a square $S$ of presheaves that has the corresponding square $S_U$ of $U$-presheaves for a set $U$ that contains the sets $Mor_{\cal C}(X,Y)$ for all $X,Y\in Ob({\cal C})$ is a pullback if and only if $S_U$ is a pullback in $PreShv({\cal C},U)$.

\comment{One can define a different concept, let us call it presheaf', that will not require $U$ in its definition and that will be for most purposes equivalent to the concept of a presheaf. Equivalence here is understood informally, in the sense that it can be used anywhere where the concept of a presheaf is used. 

Such a definition would start not with a function $F:Ob(C)\sr U$ but with a set $FF$ and a function $ff:FF\sr Ob(C)$. The connection to the usual concept is obtained by defining $U$-functions $ff$ as the ones such that $ff^{-1}(X)\in U$ for all $X$.  To a presheaf' $ff$ where $ff$ is a $U$-function one associates a functor $C\sr Sets(U)$ given on objects by $X\mapsto ff^{-1}(X)$. 

It is easy to see that in general one does not get all possible functors $C\sr Sets(U)$ from presheaves' because the sets $ff^{-1}(X)$ must be disjoint for different $X$. However, properly defined category of presheaves' will be equivalent to the category of presheaves defined as functors. 

This is well known, but to the best of my knowledge not written up anywhere remaining a mathematical folklore. }
\end{remark}
\begin{remark}\rm\llabel{2016.11.08.rem1}
In UniMath the choice of universes is a necessity early on in the theory. Type theory does not provide any means  for unbounded quantification. Therefore, to quantify over C-systems as it is required in theorems that start with ``for any C-system'' one has to quantify over some type of sets of which the sets of objects and morphisms of these C-systems are elements. To construct such a type of sets one has first to choose a universe $U$ of types.  

To define the category of sets of which this type of sets is the type of objects requires choosing a second universe of types $U'$ and an element $u$ of $U'$ such that $U=El(u)$ where $El$ is the constructor that associates to an element of a universe the corresponding type. Therefore, to define presheaves of sets in the UniMath as functors to a category of sets one has to work with two universes. One can, alternatively, define the type of presheaves directly, using $U$ only as a type without having a need for an element $u$ of a universe such that $U=El(u)$. Such direct definition is somewhat analogous to the definition that we provide in the ZF but uses 
the inherent mechanisms of working with families that exist in all dependent type theories instead of our set-theoretic definition of a family given above. 
\end{remark}

\subsection{Products of families of types and $(\Pi,\lambda)$-structures}
\label{Sec.2.3}

The ``products of families of types'' structure on a C-system is defined in \cite[pp.3.37 and 3.41]{Cartmell0} and studied further in \cite[p.71]{Streicher}. We will call it here the Cartmell-Streicher structure. Let us recall its definition. We write $Ob$ for $Ob(CC)$, $\wOb$ for the set
$$\wOb=\{o\in Mor(CC)\,|\,o\in sec(p_{codom(o)}),\,\,codom(o)>pt\}$$
and $\partial:\wOb\sr \Ob$ for the function $\partial(o)=codom(o)$.

The notation $Ob_{\ge 2}$ below refers to the set of $B\in Ob$ such that $l(B)\ge 2$. Everywhere below, for $B\in Ob_{\ge 2}$ we let $A=ft(B)$ and $\Gamma=ft^2(B)$.
\begin{definition}
\llabel{2015.03.17.def1}
A Cartmell-Streicher structure on a C-system $CC$ is a collection of data of the form
\begin{enumerate}
\item a function $\PPi:Ob_{\ge 2}\sr Ob$ such that for any $B\in Ob_{\ge 2}$ one has 
\begin{enumerate}
\item $ft(\PPi(B))=\Gamma$,
\item for any $f:\Gamma'\sr \Gamma$ one has $f^*(\PPi(B))=\PPi(f^*(B))$,
\end{enumerate}
\item for any $B\in Ob_{\ge 2}$ a morphism $Ap_B:A\times_{\Gamma}\PPi(B)\sr B$\,\footnote{In \cite[Definition 1.13, p.71]{Streicher} these morphisms are denoted $eval_B$.} such that
\begin{enumerate}
\item $Ap_B$ is a morphism over $A$,
\item for any $f:\Gamma'\sr \Gamma$ one has $f^*(Ap_B)=Ap_{f^*(B)}$,
\item the function $\lambda inv_{B}:\partial^{-1}(\PPi(B))\sr \partial^{-1}(B)$ given by
$$\lambda inv_B(s)=p_A^*(s)\circ Ap_B$$
is a bijection.
\end{enumerate} 
\end{enumerate} 
\end{definition}
The objects appearing in this definition can be seen in the following diagram
\begin{eq}\llabel{2016.05.16.eq10}
\begin{CD}
B @<Ap_B<< A\times_{\Gamma}\PPi(B) @>>> \PPi(B)\\
@Vp(B,A)VV @Vp_A^*(s)\uparrow VV @Vs\uparrow V p(\PPi(B),\Gamma)V\\
A @= A @>p_A>> \Gamma
\end{CD}
\end{eq}
The proper formulation of Definition \ref{2015.03.17.def1} requires a preliminary lemma, as we have done with Lemma \ref{2016.08.14.l2} before the definition of sets $AllAp_2^{\Pi}$. Here we provide this lemma after the definition.
\begin{lemma}\llabel{2016.08.14.l1}
Let $\PPi:Ob_{\ge 2}\sr Ob$ be a function as in Definition \ref{2015.03.17.def1}(1), $B\in Ob_{\ge 2}$ and $Ap_B:A\times_{\Gamma}\PPi(B)\sr B$ a morphism. Assume that $Ap_B$ is a morphism over $A$, then one has:
\begin{enumerate}
\item $Ap_B$ is a morphism over $\Gamma$,
\item For $s\in \partial^{-1}(\PPi(B))$ the morphism $\lambda inv_B(s)=p_A^*(s)\circ Ap_B$ is defined and belongs to $\partial^{-1}(B)$.
\end{enumerate}
\end{lemma}
\begin{proof}
The first assertion follows from Lemma \ref{2016.08.18.l2}. 

To show that $\lambda inv_B(s)$ is defined we need to check that $p_A^*(s)$ is defined, i.e., that $s$ is a morphism over $\Gamma$ and that $codom(p_A^*(s))=dom(Ap_B)$. We have
$$dom(s)=ft(\partial(s))=ft(\PPi(B))=\Gamma$$
and therefore $s$ is a morphism over $\Gamma$ by Lemma \ref{2016.09.21.l1}(1) and
$$codom(p_A^*(s))=p_A^*(codom(s))=A\times_{\Gamma}\PPi(B)=dom(Ap_B)$$

To prove that $\lambda inv_B(s)$ belongs to $\partial^{-1}(B)$ we need to show that $\lambda inv_B(s)\in \wOb(CC)$ and $codom(\lambda inv_B(s))=B$. The second statement is obvious. To prove the first one we need to show that $\lambda inv_B(s)\circ p(B,A)=Id_{A}$. We have
$$\lambda inv_B(s)\circ p(B,A)=p_A^*(s)\circ Ap_B\circ p(B,A)=$$$$
p_A^*(s)\circ p(A\times_{\Gamma}\PPi(B),A)=Id_{A}$$
where the second equality follows from the assumption that $Ap_B$ is a morphism over $B$ and the third equality follows from Lemma \ref{2016.09.21.l1}(2) for $f=p_A$. 

The lemma is proved. 
\end{proof}
\begin{definition}
\llabel{2015.03.09.def1}\llabel{2015.03.09.def2}
Let $CC$ be a C-system. A pre-$(\Pi,\lambda)$-structure on $CC$ is a pair of morphisms (natural transformations) of presheaves 
$$\Pi:\Ob_2\sr \Ob_1$$
$$\lambda:\wtOb_2\sr \wtOb_1$$
such that the square
\begin{eq}
\llabel{2015.03.09.eq1}
\begin{CD}
\wtOb_2@>\lambda>> \wtOb_1\\
@V\partial VV @VV \partial V\\
\Ob_2 @>\Pi>> \Ob_1
\end{CD}
\end{eq}
commutes.  

A pre-$(\Pi,\lambda)$-structure is called a $(\Pi,\lambda)$-structure if the square (\ref{2015.03.09.eq1}) is a pullback.
\end{definition}
\begin{remark}\rm\llabel{2016.09.25.rem1}
The category of presheaves of sets can be given the structure of a category with fiber products (see \cite[Appendix]{Pilambda} for the precise definition and for the notations used below) using the standard structure of a category with fiber products on the category of sets.

Then any pre-$(\Pi,\lambda)$-structure on $CC$ defines a morphism 
\begin{eq}
\llabel{2016.09.25.eq2}
(\partial\times\lambda)^{\Pi,\partial}:\wtOb_2\sr (\Ob_2,\Pi)\times_{\Ob_1}(\wtOb_1,\partial)
\end{eq}
which is an isomorphism if and only if this pre-$(\Pi,\lambda)$-structure is a $(\Pi,\lambda)$-structure. 

Therefore, there is a bijection between $(\Pi,\lambda)$-structures and $(\Pi,\lambda,app)$-structures where $\Pi$ and $\lambda$ form a pre-$(\Pi,\lambda)$-structure and $app$ is a morphism that is both a right and a left inverse to (\ref{2016.09.25.eq2}).   

One also obtains interesting structures by specifying in addition to $\Pi$ and $\lambda$ a morphism $app$ over $\Ob_2$ that is only a left or only a right inverse to (\ref{2016.09.25.eq2}). 

The $(\Pi,\lambda)$-structures are connected to the $(\Pi,\lambda,app,\beta,\eta)$-system of inference rules. 

The $(\Pi,\lambda,app)$-structures where $app$ is a right inverse to (\ref{2016.09.25.eq2}), that is, a morphism in the opposite direction such that
$$app\circ (\partial\times\lambda)^{\Pi,\partial}=Id_{(\Ob_2,\Pi)\times_{\Ob_1}(\wtOb_1,\partial)}$$
correspond to the $(\Pi,\lambda,app,\eta)$-system of inference rules.  

The $(\Pi,\lambda,app)$-structures where $app$ is a left inverse to (\ref{2016.09.25.eq2}), that is, a morphism in the opposite direction such that
$$(\partial\times\lambda)^{\Pi,\partial}\circ app=Id_{\wtOb_2}$$
correspond to the $(\Pi,\lambda,app,\beta)$-systems of inference rules.

Precise formulations and proofs of these correspondences require an algebraic theory of inference rules that lies outside the scope of the present paper. 

Syntactically, even a pre-$(\Pi,\lambda)$-structure may have different forms. For example, given the $\Pi$-operations defined by the scheme
$$
\frac{\Gamma,x:A\rhd B\,type}{\Gamma\rhd \Pi(A,x.B)\,type}
$$
we can have $\lambda$-operations defined by either of the two schemes
$$
\frac{\Gamma,x:A\rhd o:B}{\Gamma\rhd \lambda(A,x.o):\Pi(A,x.B)}\spc\spc{\rm or}\spc\spc  \frac{\Gamma,x:A\rhd o:B}{\Gamma\rhd \lambda(x.o):\Pi(A,x.B)}
$$
In both cases we obtain a pre-$(\Pi,\lambda)$-structure. The choice of the syntactic form for $\lambda$ may affect whether the resulting operation is in an appropriate sense free, but not whether it forms a part of a pre-$(\Pi,\lambda)$-structure. 

Operation $app$ that directly corresponds to a morphism in the direction opposite to (\ref{2016.09.25.eq2}) can have the syntactic form
\begin{eq}
\llabel{2016.09.29.eq1}
\frac{
\begin{array}{l}
\Gamma,x:A\rhd B\,type\\
\Gamma\rhd f:\Pi(A,x.B)
\end{array}}
{\Gamma,x:A\rhd app(f,x):B}
\end{eq}
or the form with $app(A,x.B,f,x)$ or any of the two intermediate forms. 

To give a precise meaning to these comments about the syntactic forms one needs a theory of syntactic C-systems. The basics of such a theory can be found in \cite{CandJf}.
\end{remark}

A homomorphism of C-systems with (pre-)$(\Pi,\lambda)$-structures is defined in \cite{Pilambda}.

In this section we construct a solution for the following problem.
\begin{problem}\llabel{2016.08.10.prob2fromold}
Let $CC$ be a C-system. To construct a bijection between the set of Cartmell-Streicher structures and the set of $(\Pi,\lambda)$-structures on $CC$.
\end{problem}
Later in the paper we will show how to construct and, sometimes, fully classify, $(\Pi,\lambda)$-structures on C-systems of the form $CC({\cal C},p)$. 

Constructing bijections is often very ``expensive'' in the sense of the time and effort required. This fact will be well illustrated by the construction of this section. 

A structure of Cartmell-Streicher on $CC$ can be seen as a pair $(\PPi,Ap)$ where $\PPi$ is a function $Ob_{\ge 2}\sr Ob$ satisfying certain conditions and $Ap$ is a function $Ob_{\ge 2}\sr Mor$ satisfying another set of conditions that depend on $\PPi$.

A $(\Pi,\lambda)$-structure is a pair $(\Pi,\lambda)$ where $\Pi$ is a morphism of presheaves $\Ob_2\sr \Ob_1$ and $\lambda$ is a morphism of presheaves $\wtOb_2\sr \wtOb_1$ satisfying certain conditions that depend on $\Pi$. 

Substituting $i=2$ and $j=1$ in Construction (\ref{2016.06.17.constr1}) we obtain a bijection $\Phi$ from the set of morphisms of presheaves on $CC$ of the form $\Pi:\Ob_2\sr \Ob_1$ to the set of functions $\PPi:Ob_{\ge 2}\sr Ob$ satisfying the conditions of Definition \ref{2015.03.17.def1}(1). 

Let $All\lambda_1^{\Pi}$ be the set of morphisms $\lambda:\wtOb_2\sr \wtOb_1$ that make the square (\ref{2015.03.09.eq1}) a pullback. 

Let $AllAp_1^\PPi$ be the set of functions $Ap:Ob_{\ge 2}\sr Mor$ that satisfy the conditions of Definition \ref{2015.03.17.def1} relative to $\PPi$. 

It remains to construct, for any morphism of presheaves $\Pi:\Ob_2\sr \Ob_1$,  a bijection of the form $All\lambda_1^{\Pi}\sr AllAp_1^{\Phi(\Pi)}$. 

Our bijection will be the composition of three bijections
\begin{eq}\llabel{2016.08.14.eq4}
All\lambda_1^{\Pi}\sr All\lambda_2^{\Pi}\sr AllAp_2^{\Phi(\Pi)}\sr AllAp_1^{\Phi(\Pi)}
\end{eq}
To define the set $AllAp_2^{\Phi(\Pi)}$ we need a lemma.
\begin{lemma}
\llabel{2016.08.14.l2}
Let $\PPi$ be a function $Ob_{\ge 2}\sr Ob$ satisfying the conditions of Definition \ref{2015.03.17.def1}(1). Let $B\in Ob_{\ge 2}$ and let $Ap_B:\PPi(B)\times_{\Gamma}A\sr B$ be a morphism such that 
\begin{eq}\llabel{2016.08.12.eq1}
Ap_B\circ p_B=q(p_{\PPi(B)},A)
\end{eq}
then one has:
\begin{enumerate}
\item $Ap_B$ is a morphism over $\Gamma$,
\item for any $s\in \partial^{-1}(\PPi(B))$  the morphism   
\begin{eq}\llabel{2016.08.10.eq1}
\lambda inv_{B}(s)=q(s,\PPi(B)\times_{\Gamma}A)\circ Ap_B
\end{eq}
is defined and belongs to $\partial^{-1}(B)$.
\end{enumerate}
\end{lemma}
\begin{proof}
The first assertion is proved by the equalities
$$Ap_B\circ p_B=Ap_B\circ p_B\circ p_{A}=q(p_{\PPi(B)},A)\circ p_{A}=$$$$p_{\PPi(B)\times_{\Gamma}A}\circ p_{\PPi(B)}=p(\PPi(B)\times_{\Gamma}A,\Gamma)$$
where the second equality follows from (\ref{2016.08.12.eq1}) and the third one from the commutativity of the right hand side canonical square in the diagram:
$$
\begin{CD}
A @>q(s,\PPi(B)\times_{\Gamma}A)>> \PPi(B)\times_{\Gamma}A@>q(p_{\PPi(B)},A)>> A\\
@Vp_{A}VV @Vp_{\PPi(B)\times_{\Gamma}A}VV @VVp_{A}V\\
\Gamma @>s>> \PPi(B) @>p_{\PPi(B)}>> \Gamma
\end{CD}
$$

For $\lambda inv_{B}(s)$ to be defined we need to have
$$codom(p_{\PPi(B)})=\Gamma$$
$$codom(s)=ft(\PPi(B)\times_{\Gamma}A)$$
$$codom(q(s,\PPi(B)\times_{\Gamma}A))=\PPi(B)\times_{\Gamma}A$$
The first equality follows from Definition \ref{2015.03.17.def1}(1a), the second equality follows from the equality $ft(f^*(X))=codom(f)$, which is one of the axioms of a C0-system, and the assumption that $s\in \partial^{-1}(\PPi(B))$. The third equality follows directly from the form of the morphisms $q(f,X)$. 

To prove that $\lambda inv_{B}(s)$ belongs to $\partial^{-1}(B)$ we need to show that $\lambda inv_{B}(s)\in \wOb(CC)$ and that $codom(\lambda inv_{B}(s))=B$. The second equality is obvious. To prove the first fact we need to show that
$$\lambda inv_{B}(s)\circ p_B=Id_{A}$$
We have:
$$\lambda inv_{B}(s)\circ p_B=q(s,\PPi(B)\times_{\Gamma}A)\circ Ap_B\circ p_B=q(s,\PPi(B)\times_{\Gamma}A)\circ q(p_{\PPi(B)},A)=Id_{A}$$
where the second equality is by (\ref{2016.08.12.eq1}) and the third one is by the composition property of the $q(-,-)$ morphisms and the fact that $s\circ p_{\PPi(B)}=Id_{\Gamma}$. The lemma is proved.
\end{proof}

Let us now defined the sets $All\lambda_2^{\Pi}$ and $AllAp_2^\PPi$:
\begin{enumerate}
\item Let $\Pi:\Ob_2\sr\Ob_1$ be a morphism of presheaves. Then the set $All\lambda_2^{\Pi}$ is the set of double families of bijections of the form
\begin{eq}
\llabel{2016.05.20.eq4}
\lambda_{\Gamma,B}:\partial^{-1}(B)\sr \partial^{-1}(\Pi_{\Gamma}(B))
\end{eq}
parametrized by $\Gamma\in Ob$ and $B\in \Ob_2(\Gamma)$ such that for any $f:\Gamma'\sr \Gamma$ and any $B\in\Ob_2(\Gamma)$ the square 
\begin{eq}\llabel{2016.04.25.eq1b}
\begin{CD}
\partial^{-1}(B) @>\lambda_{\Gamma,B} >> \partial^{-1}(\Pi_{\Gamma}(B))\\
@Vf^*_B VV @VV f^*_B V\\
\partial^{-1}(f^*(B)) @>\lambda_{\Gamma',f^*(B)} >> \partial^{-1}(\Pi_{\Gamma'}(f^*(B)))=\partial^{-1}(f^*(\Pi_{\Gamma}(B)))
\end{CD}
\end{eq}
commutes.
\item Let $\PPi:Ob_{\ge 2}\sr Ob$ be a function satisfying the conditions of Definition \ref{2015.03.17.def1}(1). Then the set $AllAp_2^\PPi$ is the set of families of morphisms of the form 
$$Ap_{B}:\PPi(B)\times_{\Gamma}A\sr B$$
parametrized by $B\in Ob_{\ge 2}$ such that:
\begin{enumerate}
\item for any $B$, (\ref{2016.08.12.eq1}) holds,
\item for any $B$ and any morphism $f:\Gamma'\sr \Gamma$ one has $f^*(Ap_B)=Ap_{f^*(B)}$, where $f^*(Ap_B)$ is defined by Lemma \ref{2016.08.14.l2}(1).
\item for any $B$, the function 
$$\lambda inv_{B}:\partial^{-1}(\PPi(B))\sr \partial^{-1}(B)$$
defined by (\ref{2016.08.10.eq1}) is a bijection.
\end{enumerate}
\end{enumerate}

Let us construct now the bijections of the sequence (\ref{2016.08.14.eq4}).
\begin{problem}
\llabel{2016.10.03.prob1}
For a given morphism of presheaves $\Pi:\Ob_2\sr\Ob_1$ to construct a bijection $All\lambda_1^{\Pi}\sr All\lambda_2^{\Pi}$.
\end{problem}
\begin{construction}\rm\llabel{2016.10.03.constr1}
We define our bijection as the particular case of the bijection $(\Psi^N_{{\cal D},0},\Phi^N_{{\cal D},0})$ of Lemma \ref{2016.07.07.l2} corresponding to the diagram of the form (\ref{2016.06.09.eq6}) where $a=\partial$, $b=\partial$ and $P=\Pi$.
\end{construction}

For a morphism of presheaves $\Pi:\Ob_2\sr\Ob_1$, the function $\Phi(\Pi):Ob_{\ge 2}\sr Ob$ is of the form
$$\Phi(\Pi)(B)=\Pi_{\Gamma}(B)$$
\begin{problem}
\llabel{2015.03.13.prob1}
Let $CC$ be a C-system and let $\Pi:\Ob_2\sr \Ob_1$ be a morphism of presheaves.
To construct a bijection between the sets  $All\lambda_2^{\Pi}$ and $AllAp_2^{\Phi(\Pi)}$.
\end{problem}
We will construct the solution in four steps - first a function from $All\lambda_2^{\Pi}$ to $AllAp_2^{\Phi(\Pi)}$, then a function in the opposite direction and then  two lemmas proving that the first function is a left and a right inverse to the second. 

Let us denote the function $\Phi(\Pi)$ by $\PPi$. 
\begin{problem}\llabel{2016.08.08.prob1}
Let $\Pi:\Ob_2\sr \Ob_1$ be a morphism of presheaves. To construct a function
$$All\lambda_2^{\Pi}\sr AllAp_2^\PPi$$
\end{problem}
\begin{construction}
\llabel{2015.03.13.constr1}\rm
For each $B\in Ob_{\ge 2}$ we need to construct a morphism of the form
$$Ap_B:\PPi(B)\times_{\Gamma}A\sr B$$
Since $\PPi$ commutes with pullbacks, that is, Definition \ref{2015.03.17.def1}(1b) holds, we have 
$$\PPi(\PPi(B)\times_{\Gamma}B)=\PPi(B)\times_{\Gamma}\PPi(B)$$
and therefore $\delta=s_{Id_{\PPi(B)}}$, which is the diagonal, gives us an element in $\partial^{-1}(\PPi(\PPi(B)\times_{\Gamma}B))$. Applying to it the inverse of $\lambda_{\PPi(B), \PPi(B)\times_{\Gamma}B}$ we get an element 
%
\begin{eq}
\llabel{2016.05.02.eq1}
ap_B=(\lambda_{\PPi(B),\PPi(B)\times_{\Gamma}B})^{-1}(\delta)
\end{eq}
in $\partial^{-1}(\PPi(B)\times_{\Gamma}B)$:
$$
\begin{CD}
B @>q(s,p_{\PPi(B)\times_{\Gamma}B})>> \PPi(B)\times_{\Gamma}B @>q(p_{\PPi(B)},B)>> B\\
@Vs^*(ap_B)\,\uparrow VV @Vap_B\,\uparrow VV @VVV\\
A @>q(s,p_{\PPi(B)\times_{\Gamma}A})>>  \PPi(B)\times_{\Gamma}A @>q(p_{\PPi(B)},A)>> A\\
@VVV @VVV @VVV\\
\Gamma @>s>> \PPi(B) @>p_{\PPi(B)}>> \Gamma
\end{CD}
$$
Define:
$$Ap_B=ap_B\circ  q(p_{\PPi(B)},B)$$

Condition (\ref{2016.08.12.eq1}) holds because
$$Ap_B\circ p_B=ap_B\circ q(p_{\PPi(B)},B)\circ p_B=ap_B\circ p_{\PPi(B)\times_{\Gamma}B}\circ q(p_{\PPi(B)},A)=$$
$$Id_{p_{\PPi(B)\times_{\Gamma}B}}\circ q(p_{\PPi(B)},A)=q(p_{\PPi(B)},A)$$

Let us prove that morphisms $Ap_B$ satisfy the conditions of stability under  base change and of the bijectivity of the corresponding functions $\lambda inv_*$. 

Let $f:\Gamma'\sr \Gamma$ be a morphism. Let us show first that $f^*(ap_B)=ap_{f^*(B)}$. Omitting indexes of $\lambda$ for clarity, we have where , 
\begin{eq}
\llabel{2016.04.25.eq4}
\begin{CD}
f^*(ap_B)=f^*(\lambda^{-1}(\delta))=\lambda^{-1}(f^*(\delta))=\lambda^{-1}(s_{f^*(Id_{\PPi(B)})})=\\
\lambda^{-1}(s_{Id_{f^*(\PPi(B))}})=\lambda^{-1}(s_{Id_{\PPi(f^*(B))}})=ap_{f^*(B)}
\end{CD}
\end{eq}
where the second equality follows from the commutativity of (\ref{2016.04.25.eq1b}), the third from (\ref{2016.05.07.eq1}), the fourth from (\ref{2016.04.25.eq3}) and the fifth from Definition \ref{2015.03.17.def1}(1b). 

Now we have:
$$f^*(Ap_B)=f^*(ap_B\circ  q(p_{\PPi(B)},B))=f^*(ap_B)\circ f^*(q(p_{\PPi(B)},B))=ap_{f^*(B)}\circ f^*(q(p_{\PPi(B)},B))=$$$$ap_{f^*(B)}\circ q(f^*(p_{\PPi(B)}),f^*(B))=
ap_{f^*(B)}\circ q(p_{f^*(\PPi(B))},f^*(B))=$$$$ap_{f^*(B)}\circ q(p_{\PPi(f^*(B))},f^*(B))=Ap_{f^*(B)}$$
where the second equality follows from (\ref{2015.06.11.eq2}), the third equality from (\ref{2016.04.25.eq4}), the fourth equality 
from (\ref{2015.06.15.eq2}), the fifth equality from (\ref{2016.05.14.eq1}) and the sixth from Definition \ref{2015.03.17.def1}(1b).

It remains to show that the functions $\lambda inv_{B}:\partial^{-1}(\PPi(B))\sr \partial^{-1}(B)$ defined as:
$$s\mapsto q(s,\PPi(B)\times_{\Gamma}A)\circ Ap_B$$
are bijective. It is sufficient to show that the function $\lambda inv_{B}$ is inverse to the function $\lambda_{\Gamma,B}$ from at least one side as any inverse to a bijection is a bijection. 

We will show that
\begin{eq}\llabel{2016.08.10.eq3}
\lambda inv_{B}\circ \lambda_{\Gamma,B}=Id_{\partial^{-1}(B)}
\end{eq}
For simplicity of notation we omit the index $\Gamma$ of $\lambda$. We proceed in two steps. First, for $s\in \partial^{-1}(\PPi(B))$, let
$$\lambda inv'_B(s)=s^*(ap_B)=q(s,\PPi(B)\times_{\Gamma}A)^*(ap_B)$$
where the second equality is by (\ref{2015.06.11.eq6}). Let us show that $\lambda inv'_B= \lambda inv_{B}$. 
Indeed:
\begin{eq}
\llabel{2016.06.25.eq1int}
\begin{CD}
q(s,\PPi(B)\times_{\Gamma}A)^*(ap_B)=q(s,\PPi(B)\times_{\Gamma}A)^*(ap_B)\circ q(s,\PPi(B)\times_{\Gamma}B)\circ q(p_{\PPi(B)},B)=\\
q(s,\PPi(B)\times_{\Gamma}A)\circ ap\circ q(p_{\PPi(B)},B) = q(s,\PPi(B)\times_{\Gamma}A)\circ Ap_B
\end{CD}
\end{eq}
where the first equality follows from Lemma \ref{2015.06.11.l2}(2) and the assumption that $s\circ \pi_{\PPi(B)}=Id_{\Gamma}$ and the second equality from (\ref{2015.06.11.eq7}).

Now we have:
$$\lambda_B(\lambda inv'_B(s))=\lambda_B(s^*(ap))=s^*(\lambda_{\PPi(B)\times_{\Gamma}B}(ap))=s^*(s_{Id_{\PPi(B)}})=s_{s\circ Id_{\PPi(B)}}=s_s=s$$
where the second equality follows from the commutativity of (\ref{2016.04.25.eq1b}) since $s^*(\PPi(B)\times_{\Gamma}B)=B$, the third from (\ref{2016.05.02.eq1}), the fourth from the formula $s_{f\circ g}=f^*(s_g)$ and the sixth from the formula $s=s_s\circ q(ft(s),X)$ (see \cite[Def. 2.3]{Csubsystems}) since $ft(s)=Id$. This completes Construction \ref{2015.03.13.constr1}.
\end{construction}
\begin{problem}\llabel{2016.08.10.prob1}
Let $\Pi:\Ob_2\sr \Ob_1$ be a morphism of presheaves. To construct a function
$$AllAp_2^\PPi\sr All\lambda_2^{\Pi}$$
\end{problem}
\begin{construction}
\llabel{2015.03.15.constr1}\rm
For $\Gamma\in Ob$ and $B\in\Ob_2(\Gamma)$ we set:
\begin{eq}\llabel{2016.08.10.eq2}
\lambda_{\Gamma,B}=(\lambda inv_{B})^{-1}
\end{eq}
where on the right hand side $B$ is considered as an element of $Ob_{\ge 2}$ such that $ft^2(B)=\Gamma$, and
$$\lambda inv_B:\partial^{-1}(\PPi(B))\sr \partial^{-1}(B)$$
is the bijection given by the formula (\ref{2016.08.10.eq1}).

To show that the bijections that we obtain in this way commute with pullbacks, in the sense that the squares (\ref{2016.04.25.eq1b}) commute, it is sufficient to show that for $\Gamma$ and $B$ as above and $f:\Gamma'\sr \Gamma$ one has 
$$\lambda inv_B\circ f^*=f^*\circ \lambda inv_{f^*(B)}$$
Let $s\in \partial^{-1}(\PPi(B))$, then we have
$$f^*(\lambda inv_{B}(s))=f^*(q(s,\PPi(B)\times_{\Gamma}A)\circ Ap_B)=f^*(q(s,\PPi(B)\times_{\Gamma}A))\circ f^*(Ap_B)=$$
$$q(f^*(s),f^*(\PPi(B)\times_{\Gamma}A))\circ Ap_{f^*(B)}=q(f^*(s),(f^*(p_{\PPi(B)}))^*(f^*(A)))\circ Ap_{f^*(B)}=$$
$$q(f^*(s),f^*(\PPi(B))\times_{\Gamma'}(ft(f^*(B)))\circ Ap_{f^*(B)}= q(f^*(s),\PPi(f^*(B))\times_{\Gamma'}(ft(f^*(B)))\circ Ap_{f^*(B)}=$$$$\lambda inv_{f^*(B)}(f^*(s))$$
where the second equality is by (\ref{2015.06.11.eq2}), the third equality is by (\ref{2015.06.15.eq2}) and the second property of $Ap_*$, the fourth equality is by (\ref{2015.06.15.eq1}), the fifth is by (\ref{2016.05.14.eq1}) and (\ref{2016.05.06.eq1}) and finally the sixth is by the fact that the condition (1b) of Definition \ref{2015.03.17.def1} holds for $\PPi$. This completes Construction \ref{2015.03.15.constr1}.
\end{construction}

Let us denote the function of Construction \ref{2015.03.13.constr1} by $C1$ and the function of Construction \ref{2015.03.15.constr1} by $C2$.
\begin{lemma}
\llabel{2015.03.15.l1}
For $\lambda_{*,*}\in All\lambda_2^{\Pi}$ one has $C2(C1(\lambda_{*,*})_*)_{*,*}=\lambda_{*,*}$. 
\end{lemma}
\begin{proof}
This is almost a tautology, but we will provide a detailed argument for it.  We need to show that for any $\Gamma\in Ob$ and $B\in\Ob_2(\Gamma)$ we have
$$C2(C1(\lambda_{*,*}))_{\Gamma,B}=\lambda_{\Gamma,B}$$
By (\ref{2016.08.10.eq2}) we have
$$C2(C1(\lambda_{*,*}))_{\Gamma,B}=(\lambda inv_B)^{-1}$$
where $\lambda inv_*$ is defined by (\ref{2016.08.10.eq1}) with $Ap_*=C1(\lambda_{*,*})_*$, that is, it is the same family of functions as appear in (\ref{2016.08.10.eq3}) and therefore we know that 
$$C2(C1(\lambda_{*,*}))_{\Gamma,B}=\lambda_{\Gamma,B}$$
because $\lambda_{\Gamma,B}$ is a bijection and a bijection has only one left inverse.
\end{proof}
\begin{lemma}
\llabel{2015.03.15.l2}
For $Ap_*\in AllAp_2^\PPi$ one has $C1(C2(Ap_*)_{*,*})_*=Ap_*$. 
\end{lemma}
\begin{proof}
Let $Ap_*\in AllAp_2^\PPi$. Then $\lambda_{*,*}=C2(Ap_*)_{*,*}$ is the double family of functions of the form
$$\lambda_{\Gamma,B}:\partial^{-1}(B)\sr \partial^{-1}(\PPi(B))$$
parametrized by $\Gamma\in Ob$ and $B\in \Ob_2(\Gamma)$ defined by the formula
$$\lambda_{\Gamma,B}=(\lambda inv_B)^{-1}$$
where on the right hand side $B$ is considered as an element of $Ob_{\ge 2}$, and for $s\in \partial^{-1}(\PPi(B))$ one has
$$\lambda inv_B(s)=q(s,\PPi(B)\times_{\Gamma}A)\circ Ap_B$$
As before, let $\delta=s_{Id_{\PPi(B)}}$Next we have
$$C1(\lambda_{*,*})_B=ap_B\circ q(p_{\PPi(B)},B)=(\lambda_{\PPi(B),\PPi(B)\times_{\Gamma}B})^{-1}(\delta)\circ q(p_{\PPi(B)},B)=$$
$$\lambda inv_{\PPi(B)\times_{\Gamma}B}(\delta)\circ q(p_{\PPi(B)},B)=$$
$$q(\delta, (\PPi(\PPi(B)\times_{\Gamma}B))\times_{\Gamma}(ft(\PPi(B)\times_{\Gamma}B)))\circ Ap_{\PPi(B)\times_{\Gamma}B}\circ q(p_{\PPi(B)},B)$$
Using that $\PPi(\PPi(B)\times_{\Gamma}B)=\PPi(B)\times_{\Gamma}\PPi(B)$ and $ft(\PPi(B)\times_{\Gamma}B)=\PPi(B)\times_{\Gamma}A$ we obtain, as our goal, the equality:
\begin{eq}\llabel{2016.04.26.eq1}
q(\delta,(\PPi(B)\times_{\Gamma}\PPi(B))\times_{\Gamma}( \PPi(B)\times_{\Gamma}A)) \circ Ap_{\PPi(B)\times_{\Gamma}B}\circ q(p_{\PPi(B)},B)=Ap_B
\end{eq}
For any $f:\Gamma'\sr \Gamma$ we have:
$$Ap_{f*(B)}\circ q(f,B)=f^*(Ap_B)\circ q(f,B)=q(f,\PPi(B)\times_{\Gamma}A)\circ Ap_B$$
where the first equality is by stability of $Ap$ under pullbacks (condition (b) in the definition of the set $AllAp_2^\PPi$) and the second by (\ref{2015.06.11.eq7}). Applying it to (\ref{2016.04.26.eq1}) and $f=p_{\PPi(B)}$ we get:
\begin{eq}\llabel{2016.04.26.eq2}q(\delta,(\PPi(B)\times_{\Gamma}\PPi(B))\times_{\Gamma}(\PPi(B)\times_{\Gamma}A)) \circ q(p_{\PPi(B)},\PPi(B)\times_{\Gamma}A)\circ Ap_B=Ap_B
\end{eq}
Next we have
$$(\PPi(B)\times_{\Gamma}\PPi(B))\times_{\Gamma}(\PPi(B)\times_{\Gamma}A)=(p_{\PPi(B)\times_{\Gamma}\PPi(B)}\circ p_{\PPi(B)})^*(A)=$$$$(q(p_{\PPi(B)},\PPi(B))\circ p_{\PPi(B)})^*(A)=q(p_{\PPi(B)},\PPi(B))^*(\PPi(B)\times_{\Gamma}A)$$
where the first equality is by Lemma \ref{2015.06.11.l2}(2), the second by the commutativity of the canonical squares in C-systems or by the commutativity clause of Lemma \ref{2015.06.15.l1} and the third is again by Lemma \ref{2015.06.11.l2}(2). 

That each of these equalities is applicable can be seen from the lower square of the following commutative diagram:
$$
\begin{CD}
\PPi(B)\times_{\Gamma}A @>>> \PPi(B)\times_{\Gamma}(\PPi(B)\times_{\Gamma}A) @>q(p_{\PPi(B)}, \PPi(B)\times_{\Gamma}A)>> \PPi(B)\times_{\Gamma}A\\
@VVV @VVV @VVV\\
\PPi(B) @>\delta>> \PPi(B)\times_{\Gamma}\PPi(B) @>q(p_{\PPi(B)},\PPi(B))>> \PPi(B)\\
@. @Vp_{\PPi(B)\times_{\Gamma}\PPi(B)}VV @VVp_{\PPi(B)}V\\ 
{} @. \PPi(B) @>p_{\PPi(B)}>> \Gamma
\end{CD}
$$
We can now rewrite (\ref{2016.04.26.eq2}) as 
\begin{eq}
\llabel{2016.05.04.eq1}
q(\delta,q(p_{\PPi(B)},\PPi(B))^*(\PPi(B)\times_{\Gamma}A))\circ q(p_{\PPi(B)},\PPi(B)\times_{\Gamma}A)\circ Ap_B=Ap_B
\end{eq}

Computing the composition of the first two morphisms we get
\begin{eq}
\llabel{2016.05.04.eq2}
\begin{CD}
q(\delta,q(p_{\PPi(B)},\PPi(B))^*(\PPi(B)\times_{\Gamma}A))\circ q(p_{\PPi(B)},\PPi(B)\times_{\Gamma}A)=\\
q(\delta,q(p_{\PPi(B)},\PPi(B))^*(\PPi(B)\times_{\Gamma}A))\circ q(q(p_{\PPi(B)},\PPi(B)),\PPi(B)\times_{\Gamma}A)=\\
q(\delta\circ q(p_{\PPi(B)},\PPi(B)),\PPi(B)\times_{\Gamma}A)=q(Id_{\PPi(B)},\PPi(B)\times_{\Gamma}A)=\\
=Id_{\PPi(B)\times_{\Gamma}A}
\end{CD}
\end{eq}
where the first equality is by (\ref{2015.06.11.eq6}), the second equality by Lemma \ref{2015.06.11.l2}(2) for $f=\delta$ and $g=q(p_{\PPi(B)},\PPi(B))$ and the fourth equality by Lemma \ref{2015.06.11.l2}(1). The third equality follows from the formula $s_f\circ q(ft(f),codom(f))=f$ (see \cite[Definition 2.3(3)]{Csubsystems}) since $ft(Id_{\PPi(B)})=Id_{\PPi(B)}\circ p_{\PPi(B)}=p_{\PPi(B)}$.

Finally, (\ref{2016.05.04.eq2}) implies (\ref{2016.05.04.eq1}), which completes the proof of the lemma and with it our construction for Problem \ref{2015.03.13.prob1}.
\end{proof}

To construct a solution for the Problem \ref{2016.08.10.prob2fromold} it remains to construct a solution to the following problem.
\begin{problem}\llabel{2016.08.12.prob1}
Let $\PPi$ be a function $Ob_{\ge 2}\sr Ob$ satisfying the conditions (1a), (1b) of Definition \ref{2015.03.17.def1}. To construct a bijection between the set $AllAp_2^\PPi$ and the set $AllAp_1^\PPi$.
\end{problem}
\begin{construction}\rm\llabel{2016.08.12.constr1}
Recall that 
\begin{enumerate}
\item $AllAp_1^\PPi$ is the set of families of morphisms of the form
$$Ap_B:A\times_{\Gamma}\PPi(B)\sr B$$
parametrized by $B\in Ob_{\ge 2}$ such that the following conditions hold:
\begin{enumerate}
\item for any $B$, $Ap_B$ is a morphism over $A$,
\item for any $f:\Gamma'\sr \Gamma$ one has $f^*(Ap_B)=Ap_{f^*(B)}$, where $f^*(Ap_B)$ is defined in view of Lemma \ref{2016.08.14.l1}(1),
\item let $\lambda inv_{B}:\partial^{-1}(\PPi(B))\sr \partial^{-1}(B)$ be the function defined
by the formula
$$\lambda inv_B(s)=p_{A}^*(s)\circ Ap_B$$
in view of Lemma \ref{2016.08.14.l1}(2), then for any $B$, $\lambda inv_B$
is a bijection.
\end{enumerate} 
\item $AllAp_2^\PPi$ is the set of families of morphisms of the form
$$Ap_B:\PPi(B)\times_{\Gamma} A\sr B$$
parametrized by $B\in Ob_{\ge 2}$ such that the following conditions hold:
\begin{enumerate}
\item for any $B$ one has $Ap_B\circ p_B=q(p_{\PPi(B)},A)$, i.e., (\ref{2016.08.12.eq1}) holds,
\item for any morphism $f:\Gamma'\sr \Gamma$ one has $f^*(Ap_B)=Ap_{f^*(B)}$, where $f^*(Ap_B)$ is defined in view of Lemma \ref{2016.08.14.l2}(1),
\item let $\lambda inv_{B}:\partial^{-1}(\PPi(B))\sr \partial^{-1}(B)$ be the function
defined by the formula
$$\lambda inv_{B}(s)=q(s,\PPi(B)\times_{\Gamma}A)\circ Ap_B$$
in view of Lemma \ref{2016.08.14.l2}(2), then for any $B$, $\lambda inv_B$ is a bijection.
\end{enumerate}
\end{enumerate}
Let $X$ be the set of (all) families of morphisms of the form 
$$A\times_{\Gamma}\PPi(B)\sr B$$
parametrized by $B\in Ob_{\ge 2}$ and $Y$ the set of (all) families of morphisms of the form 
$$\PPi(B)\times_{\Gamma} A\sr B$$
also parametrized by $B\in Ob_{\ge 2}$. 

Let $\Phi:X\sr Y$ be the function that maps $Ap^2_*\in X$ to the family $\Phi(Ap^2_*)_*$ where for $B\in Ob_{\ge 2}$ one has
$$\Phi(Ap^2_*)_B=exch(A,\PPi(B);\Gamma)\circ Ap^2_B$$
and $\Psi:Y\sr X$ the function that maps $Ap^1_*\in Y$ to the family $\Psi(Ap^1_*)_*$ where for $B\in Ob_{\ge 2}$ one has
$$\Psi(Ap^1_*)_B=exch(\PPi(B),A;\Gamma)\circ Ap^1_B$$
Because of (\ref{2016.05.20.eq3}) we have $\Phi\circ \Psi=Id_X$ and $\Psi\circ \Phi=Id_Y$, that is, $\Phi$ and $\Psi$ are mutually inverse bijections between $X$ and $Y$.

We have $AllAp^\PPi_2\subset X$ and $AllAp^\PPi_1\subset Y$. It remains to show that
\begin{eq}\llabel{2016.08.16.eq1}
\Phi(AllAp^\PPi_2)\subset AllAp^\PPi_1
\end{eq}
\begin{eq}\llabel{2016.08.16.eq2}
\Psi(AllAp^\PPi_1)\subset AllAp^\PPi_2
\end{eq}
Then Lemma \ref{2016.06.09.l1} implies that the functions
$$\Phi_0:AllAp^\PPi_2\sr AllAp^\PPi_1$$
$$\Psi_0:AllAp^\PPi_1\sr AllAp^\PPi_2$$
defined by $\Phi$ and $\Psi$ are mutually inverse bijections.

Let us prove (\ref{2016.08.16.eq1}). Let $Ap^2_*\in AllAp^\PPi_2$. Let us denote the family $\Phi(Ap^2_*)_*$ by $Ap^1_*$. For $B\in Ob_{\ge 2}$ we have
$$Ap^1_B\circ p_B=exch(A,\PPi(B);\Gamma)\circ Ap^2_B\circ p_B=$$$$exch(A,\PPi(B);\Gamma)\circ q(p_{\PPi(B)},A)=p(A\times_{\Gamma}\PPi(B)),A)$$
where the third equality is by (\ref{2016.06.18.eq6}). We conclude that $Ap^1_B$ is a morphism over $A$.

Let $f:\Gamma'\sr \Gamma$ be a morphism. Then one has
$$f^*(Ap^1_B)=
f^*(exch(A,\PPi(B);\Gamma)\circ Ap^2_B)=$$$$
f^*(exch(A,\PPi(B);\Gamma))\circ f^*(Ap^2_B)=
exch(f^*(A),f^*(\PPi(B));\Gamma')\circ Ap^2_{f^*(B)}=$$$$ exch(f^*(A),f^*(\PPi(B));ft^2(f^*(B)))\circ Ap^2_{f^*(B)}=$$$$
exch(ft(f^*(B)),\PPi(f^*(B)));ft^2(f^*(B)))\circ Ap^2_{f^*(B)}=
Ap^1_{f^*(B)}$$
where the second equality is by Lemma \ref{2016.08.14.l3}(3), the fact that $Ap^2_B$ is a morphism over $\Gamma$ and Lemma \ref{2016.08.18.l1}(2), the third equality is by Lemma \ref{2016.08.18.l1}(3) and the assumption that $f^*(Ap^2_B)=Ap^2_{f_*(B)}$, the fourth equality is by Lemma \ref{2015.06.15.l1}, including (\ref{2016.04.30.eq2}), and the fifth equality is by (\ref{2016.05.06.eq1}) and the assumption that $f^*(\PPi(B))=\PPi(f^*(B))$. We conclude that the second property of $Ap^1_B$ holds.

Let $\lambda inv^1_B$ be the function $\partial^{-1}(\PPi(B))\sr\partial^{-1}(B)$ defined, because of the two properties of $Ap^1_B$ that we have proved, by the formula
$$\lambda inv^1_B(s)=p_{A}^*(s)\circ Ap^1_B$$
and $\lambda inv^2_B$ the function $\partial^{-1}(\PPi(B))\sr\partial^{-1}(B)$ defined by the formula
$$\lambda inv^2_B(s)=q(s,\PPi(B)\times_{\Gamma}A)\circ Ap^2_B$$
We know that this latter function is a bijection.

We need to show that $\lambda inv^1_B$ is a bijection. We have
$$\lambda inv^1_B(s)=p_{A}^*(s)\circ Ap^1_B=p_{A}^*(s)\circ exch(A,\PPi(B);\Gamma)\circ Ap^2_B=$$$$exch(A,\Gamma;\Gamma)\circ q(s,\PPi(B)\times_{\Gamma}A)\circ Ap^2_B=$$$$exch(A,\Gamma;\Gamma)\circ \lambda inv^2_B(s)=\lambda inv^2_B(s)$$
where the third equality is by (\ref{2016.05.20.eq1}) and the fifth equality is by (\ref{2016.05.18.eq5}). 

Therefore $\lambda inv^1_B=\lambda inv^2_B$ and also is a bijection. We conclude that the third property of $Ap^1_B$ holds.

This completes the proof of (\ref{2016.08.16.eq1}).

Let us prove (\ref{2016.08.16.eq2}). Let $Ap^1_*\in AllAp^\PPi_1$. Let us denote the family $\Psi(Ap^1_*)_*$ by $Ap^2_*$. For $B\in Ob_{\ge 2}$ we have
$$Ap^2_B\circ p_B=exch(\PPi(B),A;\Gamma)\circ Ap^1_B\circ p_B=$$$$exch(\PPi(B),A;\Gamma)\circ p(A\times_{\Gamma}\PPi(B)),A)=q(p_{\PPi(B)},A)$$
where the third equality is by (\ref{2016.06.18.eq6}). We conclude that the first property of $Ap^2_B$ holds.

Let $f:\Gamma'\sr \Gamma$ be a morphism. Then one has
$$f^*(Ap^2_B)=
f^*(exch(\PPi(B),A;\Gamma)\circ Ap^1_B)=$$$$
f^*(exch(\PPi(B),A;\Gamma))\circ f^*(Ap^1_B)=
exch(f^*(\PPi(B)),f^*(A);\Gamma')\circ Ap^1_{f^*(B)}=$$$$ exch(f^*(\PPi(B)),f^*(A);ft^2(f^*(B)))\circ Ap^1_{f^*(B)}=$$$$ exch(\PPi(f^*(B))),ft(f^*(B));ft^2(f^*(B)))\circ Ap^1_{f^*(B)}=
Ap^2_{f^*(B)}$$
where the second equality is by Lemma \ref{2016.08.14.l3}(3), the fact that $Ap^1_B$ is a morphism over $\Gamma$ and Lemma \ref{2016.08.18.l1}(2), the third equality is by Lemma \ref{2016.08.18.l1}(3) and the assumption that $f^*(Ap^1_B)=Ap^1_{f_*(B)}$, the fourth equality is by Lemma \ref{2015.06.15.l1}, including (\ref{2016.04.30.eq2}), and the fifth equality is by (\ref{2016.05.06.eq1}) and the assumption that $f^*(\PPi(B))=\PPi(f^*(B))$. We conclude that the second property of $Ap^2_B$ holds.

Let $\lambda inv^2_B$ be the function $\partial^{-1}(\PPi(B))\sr\partial^{-1}(B)$ defined, because of the two properties of $Ap^2_B$ that we have proved, by the formula
$$\lambda inv^2_B(s)=q(s,\PPi(B)\times_{\Gamma}A)\circ Ap^2_B$$
and $\lambda inv^1_B$ the function $\partial^{-1}(\PPi(B))\sr\partial^{-1}(B)$ defined by the formula
$$\lambda inv^1_B(s)=p_{A}^*(s)\circ Ap^1_B$$
We know that this latter function is a bijection.

We need to show that $\lambda inv^2_B$ is a bijection. We have
$$\lambda inv^2_B(s)=q(s,\PPi(B)\times_{\Gamma}A)\circ Ap^2_B=q(s,\PPi(B)\times_{\Gamma}A)\circ exch(\PPi(B),A;\Gamma)\circ Ap^1_B=$$$$exch(\Gamma,A;\Gamma)\circ p_{A}^*(s)\circ Ap^1_B=$$$$exch(A,\Gamma;\Gamma)\circ \lambda inv^1_B(s)=\lambda inv^1_B(s)$$
where the third equality is by (\ref{2016.05.20.eq2}) and the fifth equality is by (\ref{2016.05.18.eq5}). 

Therefore $\lambda inv^2_B=\lambda inv^1_B$ and also is a bijection. We conclude that the third property of $Ap^2_B$ holds.

This completes the proof of (\ref{2016.08.16.eq2}) and with it Construction \ref{2016.08.12.constr1}.
\end{construction}

This completes our construction for Problem \ref{2016.08.10.prob2fromold}.

\newpage

\subsection{Appendix A. Functions and families - the case of sets}
\label{App.2A}
We start with two preliminary lemmas.
\begin{lemma}
\llabel{2016.06.09.l1}
Let $\Phi:X\sr Y$ be a bijection of sets with the inverse bijection $\Psi$. Let $X_0$, $Y_0$ be subsets in $X$ and $Y$ respectively. Assume that
$$\Phi(X_0)\subset Y_0$$
$$\Psi(Y_0)\subset X_0$$
Then the functions 
$$\Phi_0:X_0\sr Y_0$$
$$\Psi_0:Y_0\sr X_0$$
defined by $\Phi$ and $\Psi$ are mutually inverse bijections.
\end{lemma}
\begin{proof}
We have
$$\Phi_0(\Psi_0(y))=\Phi(\Psi(y))=y$$
and similarly for $\Psi_0(\Phi_0(x))$.
\end{proof}
\begin{lemma}
\llabel{2016.07.23.l4}
Let $\Phi:X\sr Y$ be a bijection of sets. Let $X_0$, $Y_0$ be subsets in $X$ and $Y$ respectively. Assume that
\begin{eq}\llabel{2016.07.23.eq5}
\Phi(X_0)\subset Y_0
\end{eq}
Then the following two conditions are equivalent:
\begin{enumerate}
\item the function 
$$\Phi_0:X_0\sr Y_0$$
defined by the inclusion (\ref{2016.07.23.eq5}), is a bijection.
\item for any $x\in X$ such that $\Phi(x)\in Y_0$ one has $x\in X_0$. 
\end{enumerate}
\end{lemma}
\begin{proof}
To show that the first condition implies the second let $\Psi_0$ be the bijection inverse to $\Phi_0$ and let $x\in X$ be such that $\Phi(x)\in Y_0$. Since $\Psi_0(\Phi(x))\in X_0$ it is sufficient to prove that $\Psi_0(\Phi(x))=x$. We have
$$\Phi(\Psi_0(\Phi(x)))=\Phi_0(\Psi_0(\Phi(x)))=\Phi(x)$$
Since $\Phi$ is, in particular, injective we conclude that $\Psi_0(\Phi(x))=x$.

To prove that the second condition implies the first one, let $\Psi$ be an inverse to $\Phi$ and let $y\in Y_0$. Then $\Phi(\Psi(y))=y\in Y_0$ and therefore $\Psi(y)\in X_0$ by our assumption. We conclude that $\Psi(Y_0)\subset X_0$ and applying Lemma \ref{2016.06.09.l1} conclude that $\Phi_0$ is a bijection.
\end{proof}

Consider a diagram $D$ of sets and functions of the form
\begin{eq}\llabel{2016.06.09.eq2}
\begin{CD}
\wt{X} @. \wt{Y}\\
@VaVV @VVbV\\
X @>f>> Y
\end{CD}
\end{eq}
Let $Fam_D$ be the set of families of functions of the form
\begin{eq}\llabel{2016.06.09.eq3}
\wt{f}_A:a^{-1}(A)\sr b^{-1}(f(A))
\end{eq}
parametrized by $A\in X$. 

We will write such a family as $(\wt{f}_A)_{A\in X}$ with $A$ in this expression being a bound variable or, when an abbreviated notation is called for, $\wt{f}_*$. The same convention will be applied to other families. 

Let $Fun_D$ be the set of functions $\wt{f}:\wt{X}\sr \wt{Y}$ such that the square 
\begin{eq}\llabel{2016.07.07.eq4}
\begin{CD}
\wt{X} @>\wt{f}>> \wt{Y}\\
@VaVV @VVbV\\
X @>f>> Y
\end{CD}
\end{eq}
commutes.

Let $\Phi_D:Fam_D\sr Fun_D$ be the function given by the formula
\begin{eq}\llabel{2016.07.23.eq1old}
\Phi_D((\wt{f}_A)_{A\in X})(x)=\wt{f}_{a(x)}(x)
\end{eq}
and $\Psi_D:Fun_D\sr Fam_D$ the function given by the formula 
\begin{eq}\llabel{2016.07.23.eq2old}
\Psi_D(\wt{f})_A(x)=\wt{f}(x)
\end{eq}
\begin{lemma}
\llabel{2016.05.22.l1new}
The functions $\Phi_D$ and $\Psi_D$ are mutually inverse bijections.
\end{lemma}
\begin{proof}
Let $\wt{f}_*\in Fam_D$, $A\in X$ and $x\in a^{-1}(A)$, then one has
$$\Psi_D(\Phi_D(\wt{f}_*))_A(x)=\Phi_D(\wt{f}_*)(x)=\wt{f}_{a(x)}(x)$$
For $x\in a^{-1}(A)$ we have $a(x)=A$ and therefore 
$$\Phi_D\circ \Psi_D=Id_{Fam_D}$$
Next, let $\wt{f}\in Fun_D$ and $x\in \wt{X}$. Then one has
$$\Phi_D(\Psi_D(\wt{f}))(x)=(\Psi_D(\wt{f}))_{a(x)}(x)=\wt{f}(x)$$
and therefore
$$\Psi_D\circ \Phi_D=Id_{Fun_D}$$
This completes the proof of the lemma.
\end{proof}
Let $Fam_{D,0}$ be the subset of $Fam_D$ that consists of families of bijections.

Let $Fun_{D,0}$ be the subset of $Fun_D$ that consists of functions $\wt{f}$ such that the square (\ref{2016.07.07.eq4}) is a pullback.
\begin{lemma}
\llabel{2016.06.09.l2}
One has
\begin{eq}
\llabel{2016.06.09.eq5}
\begin{CD}
\Phi_D(Fam_{D,0})\subset Fun_{D,0}\spc\spc(a)\\
\Psi_D(Fun_{D,0})\subset Fam_{D,0}\spc\spc(b)
\end{CD}
\end{eq}
and the corresponding functions
$$\Phi_{D,0}:Fam_{D,0}\sr Fun_{D,0}$$
$$\Psi_{D,0}:Fun_{D,0}\sr Fam_{D,0}$$
are mutually inverse bijections.
\end{lemma}
\begin{proof}
Let us show that $\Phi_D(Fam_{D,0})\subset Fun_{D,0}$. 

Let $(X,f)\times_Y (\wt{Y},b)$ be the standard fiber product of $f$ and $b$, that is, the subset in $X\times \wt{Y}$ that consists of pairs $(A,y)$ such that $f(A)=b(y)$. Let $pr_1:(X,f)\times_Y (\wt{Y},b)\sr X$ and $pr_2:(X,f)\times_Y (\wt{Y},b)\sr \wt{Y}$ be the corresponding projections. 

Let $\wt{f}\in Fun_D$ and  $\wt{f}_*=\Psi_D(\wt{f})$. Let $g=a\times_Y \wt{f}$ be the canonical function $\wt{X}\sr (X,f)\times_Y (\wt{Y},b)$. For $A\in X$ let $g_A:a^{-1}(A)\sr pr_1^{-1}(A)$ be the function given by $g_A(x)=g(x)$ and let $pr_{2,A}:pr_1^{-1}(A)\sr b^{-1}(A)$ be the similar function defined by $pr_2$. 

One verifies easily that all functions $pr_{2,A}$ are bijections. 

Since $\wt{f}=g\circ pr_2$, for any $A\in X$ we have $\wt{f}_A=g_A\circ pr_{2,A}$. Therefore, $\wt{f}_A$ is a bijection if and only if $g_A$ is a bijection. In particular, if all functions $\wt{f}_A$ are bijections then all functions $g_A$ are bijections. This implies that $g$ is a bijection. Let $(D,\wt{f})$ be the square obtained from $D$ by adding $\wt{f}$. Then $g$ defines an isomorphism from $(D,\wt{f})$ to the canonical pullback based on $f$ and $b$ and therefore it is itself a pullback. This proves (\ref{2016.06.09.eq5}(a)). 

Let us assume now that that $(D,\wt{f})$ is a pullback. Then, by the uniqueness of pullbacks, we know that $g$ is a bijection. Then all functions $g_A$ are bijections and therefore all functions $\wt{f}_A$ are bijections. This proves (\ref{2016.06.09.eq5}(b)).

The fact that $\Phi_{D,0}$ and $\Psi_{D,0}$ are mutually inverse bijections follows now from (\ref{2016.06.09.eq5}), Lemma \ref{2016.06.09.l1} and Lemma \ref{2016.05.22.l1new}.
\end{proof}
\begin{lemma}\llabel{2016.08.02.l6}
Let $D$ be of the form (\ref{2016.06.09.eq2}). Then one has:
\begin{enumerate}
\item Let $\wt{f}_*$ be a family of functions of the form (\ref{2016.06.09.eq3}). Then the square of sets of the form (\ref{2016.07.07.eq4}) with $\wt{f}=\Phi_D(\wt{f}_*)$ is a pullback if and only if all functions $\wt{f}_A$ are bijections.
\item Let $\wt{f}$ be a function such that the square (\ref{2016.07.07.eq4}) is commutative. Then all functions $\Psi_D(\wt{f})_A$ are bijections if and only if the square (\ref{2016.07.07.eq4}) is a pullback.
\end{enumerate}
\end{lemma}
\begin{proof}
Both assertions follow from Lemma \ref{2016.06.09.l2} and Lemma \ref{2016.07.23.l4}.
\end{proof}

\subsection{Appendix B. Functions and families - the case of presheaves of sets}
\label{App.2B}

Let $\cal C$ be a category. Consider a diagram $\cal D$ of presheaves of sets on $\cal C$ of the form
\begin{eq}\llabel{2016.06.09.eq6}
\begin{CD}
\wt{F} @. \wt{G}\\
@VaVV @VVbV\\
F @>P>> G
\end{CD}
\end{eq}
For $X\in {\cal C}$ let ${\cal D}(X)$ be the corresponding diagram of sets 
\begin{eq}\llabel{2016.06.09.eq7}
\begin{CD}
\wt{F}(X) @. \wt{G}(X)\\
@Va_XVV @VVb_XV\\
F(X) @>P_X>> G(X)
\end{CD}
\end{eq}
Let $Fam_{\cal D}$ be the set of double families of functions of the form 
\begin{equation}\llabel{2016.08.04.eq1}
\wt{P}_{X,A}:a_X^{-1}(A)\sr b_X^{-1}(P_X(A))
\end{equation}
parametrized by $X\in {\cal C}$ and $A\in F(X)$. 

Let $Fun_{\cal D}$ be the set of families of functions of the form 
\begin{equation}\llabel{2016.08.04.eq2}
\wt{P}_X:\wt{F}(X)\sr \wt{G}(X)
\end{equation}
such that the squares 
\begin{eq}\llabel{2016.07.07.eq5}
\begin{CD}
\wt{F}(X) @>\wt{P}_X>> \wt{G}(X)\\
@Va_XVV @VVb_XV\\
F(X) @>P_X>> G(X)
\end{CD}
\end{eq}
commute.

Applying our construction of $\Phi_D$ and $\Psi_D$ to the diagrams ${\cal D}(X)$ we get two functions
$$\Phi_{\cal D}:Fam_{\cal D}\sr Fun_{\cal D}$$
$$\Psi_{\cal D}:Fun_{\cal D}\sr Fam_{\cal D}$$
\begin{lemma}\llabel{2016.08.02.l3}
The functions $\Phi_{\cal D}$ and $\Psi_{\cal D}$ are mutually inverse bijections.
\end{lemma}
\begin{proof}
The functions on sets of families defined by families of mutually inverse bijections are mutually inverse bijections.
\end{proof}
For any $f:Y\sr X$, the squares
\begin{eq}\llabel{2016.07.27.eq1}
\begin{CD}
\wt{F}(X) @>\wt{F}(f)>> \wt{F}(Y)\\
@Va_X VV @VVa_Y V\\
F(X) @>F(f)>> F(Y)
\end{CD}
\spc\spc
\begin{CD}
\wt{G}(X) @>\wt{G}(f)>> \wt{G}(Y)\\
@Vb_X VV @VVb_Y V\\
G(X) @>G(f)>> G(Y)
\end{CD}
\end{eq}
commute. Therefore, for every $A\in F(X)$ our construction $\Psi$ gives us a function
$$\wt{F}(f)_A:a_X^{-1}(A)\sr a_Y^{-1}(F(f)(A))$$
and a function
$$\wt{G}(f)_{P_X(A)}:b_X^{-1}(P_X(A))\sr b_{Y}^{-1}(G(f)(P_X(A)))$$
Let $Fam^N_{\cal D}$ be the subset in $Fam_{\cal D}$ of double families that are ``natural in $X$'', i.e., such that for all $f:Y\sr X$ and $A\in F(X)$ the squares
\begin{eq}\llabel{2016.07.05.eq1}
\begin{CD}
a_X^{-1}(A) @>\wt{P}_{X,A}>> b_X^{-1}(P_X(A))\\
@V\wt{F}(f)_A VV @VV \wt{G}(f)_{P_X(A)}V\\
a_Y^{-1}(F(f)(A)) @>\wt{P}_{Y,F(f)(A)}>> b_Y^{-1}(P_Y(F(f)(A)))=b_Y^{-1}(G(f)(P_X(A)))
\end{CD}
\end{eq}
where the equality reflects the fact that $P$ is a morphism of presheaves, commute.  

Let $Fun^N_{\cal D}$ be the subset in $Fun_{\cal D}$ that consists of families $\wt{P}_X$ that are natural transformations (morphisms of presheaves), i.e., such that for all $f:Y\sr X$ the squares 
\begin{eq}\llabel{2016.07.07.eq3}
\begin{CD}
\wt{F}(X) @>\wt{P}_X>> \wt{G}(X)\\
@V\wt{F}(f) VV @VV\wt{G}(f) V\\
\wt{F}(Y) @>\wt{P}_Y>> \wt{G}(Y)
\end{CD}
\end{eq}
commute. 

Note that $Fun^N_{\cal D}$ is exactly the set of morphisms of presheaves $\wt{P}$ such that the square 
\begin{eq}\llabel{2016.07.07.eq6}
\begin{CD}
\wt{F} @>\wt{P}>> \wt{G}\\
@VaVV @VVbV\\
F @>P>> G
\end{CD}
\end{eq}
commutes.
\begin{lemma}
\llabel{2016.07.07.l1}
One has
\begin{eq}\llabel{2016.07.07.eq1}
\Phi_{\cal D}(Fam^N_{\cal D})\subset Fun^N_{\cal D}
\end{eq}
\begin{eq}\llabel{2016.07.07.eq2}
\Psi_{\cal D}(Fun^N_{\cal D})\subset Fam^N_{\cal D}
\end{eq}
and the corresponding functions
$$\Phi^N_{\cal D}:Fam^N_{\cal D}\sr Fun^N_{\cal D}$$
$$\Psi^N_{\cal D}:Fun^N_{\cal D}\sr Fam^N_{\cal D}$$
are mutually inverse bijections.
\end{lemma}
\begin{proof}
The second statement follows from the first one by Lemma \ref{2016.06.09.l1}. 

Let us prove (\ref{2016.07.07.eq1}).  Let $\wt{P}_{*,*}\in Fam_{\cal D}^N$. Let $f:Y\sr X$ be a morphism. We know that the squares (\ref{2016.07.05.eq1}) commute for all $A\in F(X)$. We need to show that the square (\ref{2016.07.07.eq3}) commutes, i.e., that for any $B\in \wt{F}(X)$ we have 
$$\wt{G}(f)(\wt{P}_X(B))=\wt{P}_Y(\wt{F}(f)(B))$$
where $\wt{P}_*=\Phi_{\cal D}(\wt{P}_{*,*})$. 

We have
$$\wt{G}(f)(\wt{P}_X(B))=\wt{G}(f)(\wt{P}_{X,a_X(B)}(B))=\wt{G}(f)_{b_X(\wt{P}_{X,a_X(B)}(B))}(\wt{P}_{X,a_X(B)}(B))=$$$$\wt{G}(f)_{P_{X}(a_X(B))}(\wt{P}_{X,a_X(B)}(B))=
\wt{P}_{Y,F(f)(a_X(B))}(\wt{F}(f)_{a_X(B)}(B))=$$$$=\wt{P}_{Y,a_Y(\wt{F}(f)(B))}(\wt{F}(f)_{a_X(B)}(B))=\wt{P}_{Y,a_Y(\wt{F}(f)(B))}(\wt{F}(f)(B))=\wt{P}_Y(\wt{F}(f)(B))$$
where the fourth equality is by commutativity of (\ref{2016.07.05.eq1}) and the rest of the equalities are by definitions. 

Let us prove (\ref{2016.07.07.eq2}). Let $\wt{P}_*\in Fun^N_{\cal D}$. Let $f:Y\sr X$ be a morphism. Let $A\in F(X)$. We know that the square (\ref{2016.07.07.eq3}) commutes. We need to show that the square (\ref{2016.07.05.eq1}) commutes, i.e., that for any $B\in a_X^{-1}(A)$ we have
$$\wt{G}(f)_{P_X(A)}(\wt{P}_{X,A}(B))=\wt{P}_{Y,F(f)(A)}(\wt{F}(f)_A(B))$$
We have
$$\wt{G}(f)_{P_X(A)}(\wt{P}_{X,A}(B))=\wt{G}(f)_{P_X(A)}(\wt{P}_{X}(B))=\wt{G}(f)(\wt{P}_X(B))=\wt{P}_Y(\wt{F}(f)(B))=$$
$$\wt{P}_Y(\wt{F}(f)_A(B))=\wt{P}_{Y,F(f)(A)}(\wt{F}(f)_A(B))$$
where the third equality holds by the commutativity of (\ref{2016.07.07.eq3}) and the rest of the equalities by definitions.

The lemma is proved. 
\end{proof}
\begin{lemma}\llabel{2016.08.04.l1}
Let $\cal D$ be a diagram of morphisms of presheaves of the form (\ref{2016.06.09.eq6}). Then one has:
\begin{enumerate}
\item Let $\wt{P}_{*,*}$ be a double family of functions of the form (\ref{2016.08.04.eq1}). Then the family $\Phi_{\cal D}(\wt{P}_{*,*})_*$ is a morphism of presheaves if and only if for all $f:Y\sr X$ and $A\in F(X)$ the squares (\ref{2016.07.05.eq1}) commute.
\item Let $\wt{P}_*$ be a family of functions of the form (\ref{2016.08.04.eq2}) such that the squares (\ref{2016.07.07.eq3}) commute. Then $\wt{P}_*$ is a morphism of presheaves if and only if for all $f:Y\sr X$ and $A\in F(X)$ the squares (\ref{2016.07.05.eq1}) with $\wt{P}_{*,*}=\Psi_{\cal D}(\wt{P}_*)_{*,*}$ commute. \end{enumerate}
\end{lemma}
\begin{proof}
Both assertions follow from Lemma \ref{2016.07.07.l1} and Lemma \ref{2016.07.23.l4}.
\end{proof}
Let $Fam^N_{{\cal D},0}$ be the subset of $Fam^N_{\cal D}$ that consists of those double families where all functions $\wt{P}_{X,A}$ are bijections and $Fun^N_{{\cal D},0}$ be the subset of $Fun^N_{\cal D}$ that consists of those morphisms of presheaves $\wt{P}$ that make (\ref{2016.07.07.eq6}) a pullback. We will need the following result.
\begin{lemma}\llabel{2016.08.02.l2}
A commutative square $\cal S$ of morphisms of presheaves on a category $\cal C$ is a pullback in the category of presheaves if and only if for each $X\in {\cal C}$ the square ${\cal S}(X)$ is a pullback in the category of sets.
\end{lemma}
\begin{proof}
This fact is well known, for a proof see \cite[Theorem 7.5.2, p.52]{Schubert}.
\end{proof}
\begin{lemma}
\llabel{2016.07.07.l2}
One has
$$\Phi^N_{\cal D}(Fam^N_{{\cal D},0})\subset Fun^N_{{\cal D},0}$$
$$\Psi^N_{\cal D}(Fun^N_{{\cal D},0})\subset Fam^N_{{\cal D},0}$$
and the corresponding functions 
$$\Phi^N_{{\cal D},0}:Fam^N_{{\cal D},0}\sr Fun^N_{{\cal D},0}$$
$$\Psi^N_{{\cal D},0}:Fun^N_{{\cal D},0}\sr Fam^N_{{\cal D},0}$$
are mutually inverse bijections.
\end{lemma}
\begin{proof}
The second statement follows from the first one by Lemma \ref{2016.06.09.l1}. 

The first statement follows from Lemma \ref{2016.06.09.l2} and Lemma \ref{2016.08.02.l2}. 
\end{proof}
\begin{lemma}\llabel{2016.08.06.l3}
Let $\cal D'$ be a commutative square of morphisms of presheaves of the form (\ref{2016.07.07.eq6}) and $\cal D$ be a diagram of morphisms of presheaves of the form (\ref{2016.06.09.eq6}) obtained from $\cal D'$ by removing $\wt{P}$. Then $\cal D'$ is a pullback if and only if for all $X\in C$ and $A\in F(X)$ the function 
$$\Psi_{\cal D}(\wt{P})_{X,A}:a_X^{-1}(A)\sr b_X^{-1}(P_X(A))$$
is a bijection.
\end{lemma}
\begin{proof}
The "only if" part follows from the second inclusion of Lemma \ref{2016.07.07.l2}. The "if" part follows from the fact that $\Psi_{\cal D}^N$ is bijective by Lemma \ref{2016.07.07.l1} and Lemma \ref{2016.07.23.l4}.
\end{proof}

\subsection{Appendix C. Exchange isomorphisms}
\label{App.C}

Let $f:X\sr Z$, $g:Y\sr Z$ be two morphisms and let
$$
\begin{CD}
pb(f,g) @>pr^{f,g}_X>> X\\
@Vpr^{f,g}_Y VV @VV f V\\
Y @>g>> Z
\end{CD}
\spc\spc
\begin{CD}
pb(g,f) @>pr^{g,f}_Y>> Y\\
@Vpr^{g,f}_X VV @VV gV\\
X @>f>> Z
\end{CD}
$$
be pullbacks. Since the first square commutes we have
$$pr^{f,g}_X\circ f=pr^{f,g}_Y\circ g$$
and therefore since the second square is a pullback there exists a unique morphism
$$exch_1:pb(f,g)\sr pb(g,f)$$
such that
$$exch_1\circ pr^{g,f}_X=pr^{f,g}_X$$
$$exch_1\circ pr^{g,f}_Y=pr^{f,g}_Y$$
Applying the same reasoning with the roles of two squares exchanged we obtain a unique morphism
$$exch_2:pb(g,f)\sr pb(f,g)$$
such that
$$exch_2\circ pr^{f,g}_Y=pr^{g,f}_Y$$
$$exch_2\circ pr^{f,g}_X=pr^{g,f}_X$$
\begin{lemma}
\llabel{2016.05.18.l2}
In the notations introduced above one has
\begin{eq}
\llabel{2016.05.18.eq2}
exch_1\circ exch_2=Id_{pb(f,g)}
\end{eq}
$$exch_2\circ exch_1=Id_{pb(g,f)}$$
In particular, both $exch_1$ and $exch_2$ are isomorphisms.
\end{lemma}
\begin{proof}
It is sufficient to prove (\ref{2016.05.18.eq2}). The proof of the second equality is exactly symmetrical to the proof of the first. 

The domain and codomain of both sides of (\ref{2016.05.18.eq2})  is $pb(f,g)$. We know that $pb(f,g)$ is a pullback with projections $pr^{f,g}_X$ and $pr^{f,g}_Y$. Therefore it is sufficient to show that the compositions of the left and right hand sides of (\ref{2016.05.18.eq2}) with these projections coincide. 

We have
$$exch_1\circ exch_2\circ pr^{f,g}_X=exch_1\circ pr^{g,f}_X=pr^{f,g}_X=Id_{pb(f,g)}\circ pr^{f,g}_X$$
and similarly for the second projection. The lemma is proved. 
\end{proof}

\section{Acknowledgements}
\label{Sec.6}

I am grateful to the Department of Computer Science and Engineering of the University of Gothenburg and Chalmers University of Technology for its the hospitality during my work on the paper.  

Work on this paper was supported by NSF grant 1100938.

This material is based on research sponsored by The United States Air Force Research Laboratory under agreement number FA9550-15-1-0053. The US Government is authorized to reproduce and distribute reprints for Governmental purposes notwithstanding  any copyright notation thereon.

The views and conclusions contained herein are those of the author and should not be interpreted as necessarily representing the official policies or endorsements, either expressed or implied, of the United States Air Force Research Laboratory, the U.S. Government or Carnegie Mellon University.


\def\cprime{$'$}

\end{document}